\documentclass[a4paper,11pt]{amsart}
\usepackage[latin1]{inputenc}
\usepackage[active]{srcltx}
\usepackage{amsmath}
\usepackage{amsfonts}
\usepackage{amssymb}
\usepackage{amsthm}
\usepackage{mathrsfs}
\usepackage{graphicx}

\newcommand{\R}{{\mathbb{R}}}

\newcommand{\Z}{{\mathbb{Z}}}
\newcommand{\C}{\mathbb{C}}
\newcommand{\N}{\mathbb{N}}

\def\scrN{{\mathcal N}}

\def\vecd{{\text{\boldmath$d$}}}
\def\vecr{{\text{\boldmath$r$}}}
\def\vecx{{\text{\boldmath$x$}}}
\def\vecy{{\text{\boldmath$y$}}}
\def\vecz{{\text{\boldmath$z$}}}

\def\vecv{{\text{\boldmath$v$}}}

\def\vecm{{\text{\boldmath$m$}}}
\def\veck{{\text{\boldmath$k$}}}

\def\vece{{\text{\boldmath$e$}}}

\def\vecphi{{\text{\boldmath$\phi$}}}

\def\vec0{{\text{\boldmath$0$}}}

\newcommand{\tA}{\widetilde{A}}

\def\Re{\operatorname{Re}}

\def\B{\operatorname{B}}
\def\OO{\operatorname{O}}

\def\Prob{\operatorname{Prob}}

\def\Span{\operatorname{Span}}

\newcommand{\V}{\mathcal{V}}
\renewcommand{\mod}{\text{ mod }}
\newcommand{\tV}{\widetilde{\mathcal{V}}}

\newcommand{\col}{\: : \:}
\newcommand{\bn}{\mathbf{0}}

\newcommand{\E}{\mathbb{E}}
\newcommand{\oA}{\overline{A}}
\newcommand{\oF}{\overline{F}}
\newcommand{\ve}{\varepsilon}

\newcommand{\sfrac}[2]{{\textstyle \frac {#1}{#2}}}

\newcommand{\SL}{\mathrm{SL}}

\newcommand{\vol}{\mathrm{vol}}

\newtheorem{thm}{Theorem}[section]
\newtheorem{prop}[thm]{Proposition}
\newtheorem{lem}[thm]{Lemma}

\theoremstyle{remark}
\newtheorem{remark}[thm]{Remark} 

\numberwithin{equation}{section}

\begin{document}
%opening
\title[On the generalized circle problem for a random lattice]{On the generalized circle problem for\\ a random lattice in large dimension}

\author{Andreas Str\"ombergsson and Anders S\"odergren}
\address{Department of Mathematics, Box 480, Uppsala University, 751 06 Uppsala, Sweden\newline
\rule[0ex]{0ex}{0ex}\hspace{8pt} {\tt astrombe@math.uu.se}\newline
\newline
\rule[0ex]{0ex}{0ex} \hspace{8pt}Department of Mathematical Sciences, University of Copenhagen, Universitetsparken\newline
\rule[0ex]{0ex}{0ex}\hspace{8pt} 5, 2100 Copenhagen \O, Denmark\newline
%\rule[0ex]{0ex}{0ex}\hspace{8pt} {\tt sodergren@math.ku.dk}\newline
\rule[0ex]{0ex}{0ex}\hspace{8pt} \textit{Present address}: School of Science and Technology, Örebro University, 701 82 Örebro,\newline
\rule[0ex]{0ex}{0ex}\hspace{8pt} Sweden\newline
\rule[0ex]{0ex}{0ex}\hspace{8pt} {\tt anders.sodergren@oru.se}} 

\date{\today}
\thanks{The first author is supported by a grant from the G\"oran Gustafsson Foundation for
Research in Natural Sciences and Medicine, and also by the Swedish Research Council Grant 621-2011-3629. The second author was partially supported by a postdoctoral fellowship from the Swedish Research Council, by the National Science Foundation under agreement No.\ DMS-1128155, as well as by a grant from the Danish Council for Independent Research and FP7 Marie Curie Actions-COFUND (grant id: DFF-1325-00058)}

\begin{abstract}
%In this note we continue the investigation initiated by the second author in \cite{poisson} of the distribution of lengths of lattice vectors in random lattices of large dimension $n$.
In this note we study the error term $R_{n,L}(x)$ in the generalized circle problem for a ball of volume $x$ and a random lattice $L$ of large dimension $n$. %Generalizing a mean value formula of C.\ A.\ Rogers, we prove 
Our main result is the following functional central limit theorem: Fix an arbitrary function $f:\Z^+\to \R^+$ satisfying $\lim_{n\to\infty}f(n)=\infty$ and $f(n)=O_{\ve}(e^{\ve n})$ for every $\ve>0$. Then, the random function 
\begin{align*}
t\mapsto\frac{1}{\sqrt{2f(n)}}R_{n,L}\left(tf(n)\right)
\end{align*}
on the interval $[0,1]$ converges in distribution to one-dimensional Brownian motion as $n\to\infty$.
The proof goes via convergence of moments, and for the computations we develop a 
new version of Rogers' mean value formula from \cite{rogers1}.
For the individual $k$th moment of the variable $(2f(n))^{-1/2}R_{n,L}(f(n))$ we prove convergence to the 
corresponding Gaussian moment more generally for functions $f$ satisfying
$f(n)=O(e^{cn})$ for any fixed $c\in(0,c_k)$, where $c_k$ is a constant depending on $k$
whose optimal value we determine.
%>0$ less than a constant $c_k$, and we determine the optimal value of $c_k$ for this to hold.
\end{abstract}

\maketitle

\section{Introduction}

Gauss' circle problem is a classical problem in number theory asking for the number of integer lattice points inside a Euclidean circle of radius $t$ centered at the origin. Gauss observed that this quantity equals the area $A(t)=\pi t^2$ enclosed by the circle up to an error term of size at most $O(t)$. Hardy conjectured \cite{hardy} that the error term can be improved to $O_{\ve}(t^{1/2+\ve})$; a bound which is known to be essentially optimal. Despite efforts of many mathematicians, Hardy's conjecture remains open and the best known bound is $O_{\ve}(t^{131/208+\ve})$ due to Huxley \cite{huxley}.

In this paper we will be interested in the circle problem generalized to dimension $n$ and a general $n$-dimensional lattice $L$ of covolume $1$. We denote the space of all such lattices by $X_n$ and recall that $X_n$ can be identified with the homogeneous space $\SL(n,\Z)\backslash \SL(n,\R)$ via the correspondence $\Z^ng\leftrightarrow\SL(n,\Z)g$. As a consequence of this identification, $X_n$ inherits a right $\SL(n,\R)$-invariant probability measure $\mu_n$ originating from a Haar measure on $\SL(n,\R)$. 

Given $n\geq2$, a lattice $L\in X_n$ and a real number $x\geq0$, we let $N_{n,L}(x)$ denote the number of non-zero lattice points of $L$ in the closed ball of volume $x$ centered at the origin in $\R^n$, i.e.\ we let
\begin{equation}\label{NnLdef}
N_{n,L}(x):=\#\bigg\{\vecm\in L\setminus\{\vec0\} : |\vecm|\leq\Big(\frac{x}{V_n}\Big)^{1/n}\bigg\},
\end{equation}  
where $V_n$ denotes the volume of the unit ball in $\R^n$. We also define, for $x\geq0$, the function 
\begin{equation*}
R_{n,L}(x):=N_{n,L}(x)-x,
\end{equation*}
and formulate, for a given $L\in X_n$, the generalized circle problem as the problem of giving the best possible upper bound on $R_{n,L}(x)$ as $x\to\infty$.

In a series of papers Bentkus and G\"otze \cite{BG1,BG2} and G\"otze \cite{gotze} proved strong explicit bounds on $R_{n,L}(x)$ for an arbitrary given lattice $L\in X_n$. In particular, G\"otze proved in \cite{gotze} that $|R_{n,L}(x)|=O(x^{1-2/n})$ holds for every $L\in X_n$ when $n\geq5$. This result is best possible for all rational lattices $L\in X_n$, while for irrational lattices G\"otze proved the stronger bound $R_{n,L}(x)=o(x^{1-2/n})$ as $x\to\infty$.\footnote{Here we call a lattice $L$ irrational if the Gram matrix for every $\Z$-basis of $L$ is not proportional to a matrix with integer entries only.} However, it turns out that for most lattices (in the measure sense) one can do much better. In fact, Schmidt \cite{schmidt} proved that for any $n\geq2$ and $\mu_n$-almost every $L\in X_n$ we have $R_{n,L}(x)=O_{\ve}(x^{1/2}(\log x)^{5/2+\ve})$. This upper bound should be compared to Landau's result $R_{n,L}(x)=\Omega(x^{1/2-1/(2n)})$ (cf.\ \cite{landau}). Hence, for large $n$, Schmidt's bound is close to optimal. In this vein it should also be noted that, for $n\geq3$, \footnote{Throughout the paper, %Here and throughout, 
$\mathbb E$ will denote the expected value with respect to the measure $\mu_n$ on $X_n$.}
\begin{equation}\label{variance} 
\text{Var}(R_{n,L}(x))=\mathbb E\left(R_{n,L}(x)^2\right):=\int_{X_n}R_{n,L}(x)^2\,d\mu_n(L)\asymp x
\end{equation}
(cf., e.g., \cite[p.\ 518]{schmidt} or \cite[Lemma 3.1]{epstein2}).

In a closely related direction, the second author has recently studied the distribution of lengths of lattice vectors in a $\mu_n$-random lattice of large dimension $n$. Given a lattice $L\in X_n$, we order its non-zero vectors by increasing lengths as $\pm\vecv_1,\pm\vecv_2,\pm\vecv_3,\ldots$ and define, for each $j\geq1$,
\begin{equation*}
\mathcal V_j(L):=V_n|\vecv_j|^n.
\end{equation*} 
We stress that the first few vectors in this list, that is, the shortest non-zero vectors in $L$, encode important geometric information attached to $L$. Indeed, these short vectors play a crucial role in, for example, the lattice sphere packing problem where the quantity $2^{-n}\sup_{L\in X_n}\mathcal V_1(L)$ determines the maximal density of a lattice sphere packing in $\R^n$. In \cite{poisson}, by calculating the limits as $n\to\infty$ of mixed moments of the form
\begin{equation}\label{EXAMPLESOFMOMENTS}
\mathbb E\bigg(\prod_{j=1}^{k}N_{n,L}(x_j)\bigg)
\end{equation}
for any fixed $k\geq1$ and $0<x_1\leq x_2\leq\ldots\leq x_k$, 
%and matching these limits with the appropriate mixed moments of Poisson distributed variables, 
the following theorem is established:

\begin{thm}[S\"odergren]\label{POISSONTHEOREM}
The sequence $\{\mathcal V_j(\cdot)\}_{j=1}^{\infty}$ converges in distribution, as $n\to\infty$, to the sequence $\{ T_j\}_{j=1}^{\infty}$, where $0<T_1<T_2<T_3<\cdots$ denote the points of a Poisson process $\mathcal P=\big\{\mathcal N(x) , x\geq0\big\}$ on $\R^+$ with constant intensity $\frac{1}{2}$.  
\end{thm}

The convergence in Theorem \ref{POISSONTHEOREM} is equivalent to the convergence of all finite dimensional distributions,
i.e.\ to the fact that the truncated sequence $\{\mathcal V_j(\cdot)\}_{j=1}^{N}$ converges in distribution
to the corresponding truncated sequence $\{ T_j\}_{j=1}^{N}$, for every fixed $N\in\Z^+$.
% language used here, ``fin dim distributions'' etc: Just as in Donsker 1961, first page!
This raises the question whether %or not 
it is possible to allow for more flexibility in Theorem \ref{POISSONTHEOREM} in the sense of allowing $N=N(n)$ to grow as a function of the dimension $n$? It seems reasonable to expect that for moderately growing $N$ the Possion characteristic of the limit sequence should remain intact, but that the Poissonian behavior will eventually disappear as $N$ is allowed to grow faster. 
A first result in this direction, indicating a Poissonian behavior for
$N\leq cn$ where $c>0$ is a small absolute constant,
is proved in a recent paper by Kim \cite{kim2016} using a sieving argument
(cf.\ also \cite{kim} where the range $N\leq(n/2)^{1/2-\ve}$ was obtained).
The following result extends this range, giving an indication of Poissonian behavior for
any $N$ growing \textit{sub-exponentially} with respect to $n$.

%Previous version; changed 20161117:
%A first result in this direction, indicating Poissonian behavior for $N\leq(n/2)^{1/2-\ve}$, is proved in a %very 
%recent paper by Kim \cite{kim} using an inclusion-exclusion argument. The following %Our first 
%result extends this range, giving an %a first 
%indication of %indicating a 
%Poissonian behavior for any $N$ growing \textit{sub-exponentially} with respect to $n$.

\begin{thm}\label{Poissoncor}
Let $f:\Z^+\to\R^+$ be any function satisfying $\lim_{n\to\infty}f(n)=\infty$ and $f(n)=O_{\ve}(e^{\ve n})$ for every $\ve>0$.
Let $\scrN(x)$ be a Poisson distributed random variable with expectation $x/2$.
Then
\begin{align}\label{PoissoncorRES}
\Prob_{\mu_n}(N_{n,L}(x)\leq 2N)-\Prob(\scrN(x)\leq N)\to 0
\quad\text{as }\: n\to\infty,
\end{align}
uniformly with respect to all $N,x\geq0$ satisfying $\min(x,N)\leq f(n)$.
%uniformly with respect to $N\geq0$ and $0\leq x\leq f(n)$.
\end{thm}

We will deduce Theorem \ref{Poissoncor} from Theorem \ref{POISSONTHEOREM} combined with the following
result, %Theorem \ref{FIRSTMAINTHEOREM},
a central limit theorem for the normalized error term in the generalized circle problem
for a random lattice $L$.
% NO NEED to say the following!! (Namely I think the formulation ``Thm .. combined with the following ..'' in itself
% suggests that Theorem \ref{Poissoncor} is ``more general'' than Theorem \ref{FIRSTMAINTHEOREM}!)
%\footnote{Note that aposteriori, Theorem \ref{FIRSTMAINTHEOREM} may be viewed as a special case of Theorem \ref{Poissoncor}.}

\begin{thm}\label{FIRSTMAINTHEOREM}
Let $f:\Z^+\to \R^+$ be any function satisfying $\lim_{n\to\infty}f(n)=\infty$ and $f(n)=O_{\ve}(e^{\ve n})$ for every $\ve>0$.
Let $Z_{n}^{(\B)}$ be the random variable
\begin{align}\label{FIRSTMAINTHEOREMdef}
Z_{n}^{(\B)}:=\frac{1}{\sqrt{2f(n)}}R_{n,L}(f(n)),
\end{align}
with $L$ picked at random in $(X_n,\mu_n)$.
Then
\begin{equation*}
Z_{n}^{(\B)}\xrightarrow[]{\textup{ d }}N(0,1)\qquad\text{as }\: n\to\infty.
\end{equation*}  
\end{thm}
The ``$\B$'' in $Z_{n}^{(\B)}$ stands for ``ball''. % (centered at the origin).
In fact, the same convergence holds even if %, instead of balls, 
we consider completely general subsets
of $\R^n$ symmetric about the origin.
\\[5pt]
\noindent
\textbf{Theorem \ref{FIRSTMAINTHEOREM}'.}
\textit{Let $f:\Z^+\to \R^+$ be as in Theorem \ref{FIRSTMAINTHEOREM}, and for each $n$ let $S_n$ be a 
Borel measurable subset of $\R^n$ satisfying $\vol(S_n)=f(n)$ and $S_n=-S_n$. Set
\begin{align}\label{WZNTdef}
Z_{n}:=\frac{\#(L\cap S_n\setminus\{\bn\})-f(n)}{\sqrt{2f(n)}},
\end{align}
with $L$ picked at random in $(X_n,\mu_n)$.
Then
\begin{equation*}
Z_{n}\xrightarrow[]{\textup{ d }}N(0,1)\qquad\text{as }\: n\to\infty.
\end{equation*}}
\begin{remark}
%The statement of 
Theorem \ref{FIRSTMAINTHEOREM}' remains true if we consider
$L\cap S_n$ instead of $L\cap S_n\setminus\{\bn\}$ in \eqref{WZNTdef}, since $f(n)\to\infty$.
However, the fact that we remove 
$\bn$ in \eqref{NnLdef} is essential for Theorem \ref{Poissoncor} to hold,
namely in the case when $x$ stays
bounded as $n\to\infty$.
\end{remark}

In Theorem \ref{independentgaussians} below we generalize Theorem \ref{FIRSTMAINTHEOREM}' to the case of
$r$ pairwise disjoint subsets of $\R^n$, 
for any fixed $r\in\Z^+$,
showing that the joint distribution of the normalized counting variables approaches $r$ independent normal distributions.
In the special case of balls centered at the origin, %the error term in the generalized circle problem,
we also have the following \textit{functional} central limit theorem,
generalizing Theorem \ref{FIRSTMAINTHEOREM}:

\begin{thm}\label{brownian}
Let $f:\Z^+\to \R^+$ be any function satisfying $\lim_{n\to\infty}f(n)=\infty$ and $f(n)=O_{\ve}(e^{\ve n})$ for every $\ve>0$. Consider, for $n\in\Z^+$ and $L$ picked at random in $(X_n,\mu_n)$,
the random function
\begin{align*}
t\mapsto\widetilde Z_{n}^{(\B)}(t):=\frac{1}{\sqrt{2f(n)}}R_{n,L}\left(tf(n)\right)
\end{align*}
on the interval $[0,1]$.
Let $P_n$ denote the corresponding probability measure on the space $\mathcal D[0,1]$ of cadlag functions on $[0,1]$. Then $\widetilde Z_{n}^{(\B)}(t)$ converges in distribution to one-dimensional Brownian motion, or equivalently, $P_n$ converges weakly to Wiener measure, as $n\to\infty$.
\end{thm}

\begin{remark} %Let us mention that 
In a different direction, for $n=2$ and fixed $L\in X_2$, a result by Bleher \cite{bleher} (cf.\ also Heath-Brown \cite{heath-brown} for the case $L=\Z^2$) implies the existence of a limit distribution of $t^{-1/4}R_{2,L}(t)$ for $t$ random in $(0,T)$, as $T\to\infty$. This limit distribution is non-Gaussian; however the corresponding limit for the number of lattice points in thin annuli is Gaussian in certain situations; cf.\ \cite{hughesrudnick} and \cite{wigman}. We are not aware of any similar results in dimension $n\geq3$; cf.\ however Peter \cite{peter}.
\end{remark}

\vspace{5pt}

It is an interesting question whether the above limit results could be extended to more rapidly
growing functions $f(n)$. %, in particular to $f(n)$ of modest exponential growth.
Our proof of Theorem \ref{FIRSTMAINTHEOREM}'
goes by establishing convergence of all moments of $Z_n$.
For any \textit{fixed} moment $\E(Z_n^{\:k})$,
the method actually yields the desired limit result even for $f(n)$ of modest exponential growth;
however for more rapidly growing $f(n)$ the moment diverges (if $k\geq3$).
In the case of balls, we have determined the precise growth rate where this transition occurs:
%The result is as follows.
Set %Let us set
\begin{align}\label{CKDEF}
c_2=+\infty \quad\text{and}\quad c_k=\frac{k-1}{k-2}\log(k-1)-\log k\quad(k\geq3). %\:\text{ for } k\geq3.
\end{align}
Note that $\{c_k\}_{k\geq3}$ is a positive, strictly decreasing sequence;
its first values are $c_3=0.28768\ldots$, $c_4=0.26162\ldots$, $c_5=0.23895\ldots$,
and $c_k\sim k^{-1}\log k$ as $k\to\infty$.

\begin{thm}\label{OPTIMALFNGROWTHTHM}
Let $k\geq2$ and $0<c<c_k$,
and let $f:\Z^+\to \R^+$ be any function satisfying $\lim_{n\to\infty}f(n)=\infty$ and $f(n)=O(e^{c n})$.
For each $n$ let $S_n$ be a 
Borel measurable subset of $\R^n$ satisfying $\vol(S_n)=f(n)$ and $S_n=-S_n$,
and define $Z_n$ as in Theorem \ref{FIRSTMAINTHEOREM}'.
Then
\begin{align}\label{OPTIMALFNGROWTHTHMres}
\lim_{n\to\infty}\mathbb E\big(Z_n^{\:k}\big)=
\begin{cases}
0 & \text{ if $k$ is odd,}\\
(k-1)!! & \text{ if $k$ is even.}
\end{cases}
\end{align}
On the other hand, if $k\geq3$ and $c>c_k$, and if $f:\Z^+\to\R^+$ is any function satisfying
$f(n)\gg e^{cn}$ as $n\to\infty$,
then $\mathbb E\big(\bigl(Z_n^{(\B)}\bigr)^k\big)\to+\infty$ as $n\to\infty$.
\end{thm}

The last result %clearly 
shows in particular that the assumption of sub-exponential growth %rate of $f(n)$ 
imposed in Theorem \ref{FIRSTMAINTHEOREM}' is best possible %(i.e.\ cannot be weakened) %cannot be weakened further
% YES, ok to say Thm. 1.3', namely if best possible for BALLS then also (``by def'') best possible for GENERAL sets S_n!
for our method of proof via convergence of moments;
however the question remains open whether a limit distribution of $Z_n$ exists 
(Gaussian or not) also for more rapidly growing $f(n)$.
Theorem \ref{OPTIMALFNGROWTHTHM} shows in this regard that
%What Theorem \ref{OPTIMALFNGROWTHTHM} shows in this regard is that
%At any rate, Theorem \ref{OPTIMALFNGROWTHTHM} shows that  %What Theorem \ref{OPTIMALFNGROWTHTHM} shows is that 
any limit distribution of any subsequence of $Z_n$
is necessarily \textit{close} to the Gaussian $N(0,1)$ distribution, in the weak topology,
so long as $f(n)=O(e^{cn})$ with $c>0$ sufficiently small.

% NOT NEEDED (namely the statement can be seen by a more ``abstract'' argument as in our proof of Theorem 1.2!
%(Cf.\ the Chebyshev-Markov-Stieltjes inequalities; see e.g.\ \cite[Thm.\ 2.5.4]{akhiezer65}.)

\begin{remark}
In the setting of balls as in Theorem \ref{FIRSTMAINTHEOREM},
taking $f(n)=e^{cn}$ corresponds to counting all lattice vectors $\vecm\in L\setminus\{\bn\}$
of length $|\vecm|\lesssim e^c\sqrt{\frac n{2\pi e}}$.
%It is worth pointing out 
In this connection we note that
%Note in this vein that  %that Theorem \ref{POISSONTHEOREM} entails that
for any fixed $N$, with probability tending to one as $n\to\infty$,
the first $N$ shortest non-zero vectors $\pm\vecv_1,\dots,\pm\vecv_N$ of a random lattice $L\in X_n$
all have length $\sqrt{\frac n{2\pi e}}(1+O(\frac{\log n}n))$.
This follows e.g.\ from Theorem \ref{POISSONTHEOREM}, using the asymptotics
$V_n\sim(\frac{2\pi e}n)^{n/2}(\pi n)^{-1/2}$.
\end{remark}

\enlargethispage{15pt}

\begin{remark}
Kelmer has recently obtained a bound on the mean square of $R_{n,L}(x)$ for fixed $n\geq2$ 
and large $x$; cf.\ \cite[Thm.\ 2]{kelmer}.
This bound supports the conjecture that for almost every $L\in X_n$,
$R_{n,L}(x)\ll x^{\frac12-\frac1{2n}+\ve}$ holds as $x\to\infty$
(cf.\ also \cite{gotze98}, \cite{holmin}).
Kelmer's bound implies that if $f(n)$ grows \textit{sufficiently rapidly}
(the growth condition could be made explicit with further work),
%(the precise growth condition could be made explicit with further work),
then $Z_n^{(\B)}$ converges in distribution to $0$
as $n\to\infty$, showing that the normalization in \eqref{FIRSTMAINTHEOREMdef} 
is inappropriate in this regime.
\end{remark}

Our original motivation for studying the limit distribution of $Z_n^{(\B)}$ comes from questions concerning
the Epstein zeta function of a random lattice $L\in X_n$ as $n\to\infty$;
cf.\ \cite{SaS, epstein1, epstein2}.
Recall that for $\Re s>\frac n2$ and $L\in X_n$ the Epstein zeta function is defined by the absolutely convergent series
\begin{equation*}
E_n(L,s):=\sum_{\vecm\in L\setminus\{\vec0\}}|\vecm|^{-2s}.
\end{equation*}  
The function $E_n(L,s)$ can be meromorphically continued to $\C$ and satisfies a functional equation of
"Riemann type" relating $E_n(L,s)$ and $E_n(L^*,\frac n2-s)$. (Here $L^*$ denotes the dual lattice of $L$.)
An outstanding question from \cite{epstein2} is whether $E_n(L,s)$
for $s$ on or near the central point $s=\frac n4$, possesses,
after appropriate normalization, a limit distribution as $n\to\infty$?
This question turns out to be closely related to the behavior of the random function $\widetilde Z_n^{(\B)}(t)$,
and we expect that Theorem \ref{brownian} in this paper in combination with the methods of
\cite{epstein2} will make it possible to give an answer in the case of $s=cn$
with $c>\frac14$ tending to $\frac14$ sufficiently slowly as a function of $n$.
However in order to handle $c=\frac14$ or $c$ arbitrarily near $\frac14$, % or $c$ arbitrarily near $\frac14$,
it appears that we need a precise understanding of the limit of $\widetilde Z_n^{(\B)}(t)$
when the volume $f(n)$ is allowed to grow as rapidly as $e^{\frac12(1-\log2)n}$,
and furthermore we need to understand this distribution jointly with the corresponding distribution for the
dual lattice of $L$.
We hope to return to these matters in future work.

%A preliminary asymptotic study following \cite{epstein2} indicates that this question is crucially related to
%the behavior of the random variable $Z_n^{(\B)}$ for $f(n)$ of growth at worst $O(e^{\frac12(1-\log2)n})$.

\vspace{5pt}

The organization of the paper is as follows.
%We end the introduction with an outline of the proofs of the theorems.
As mentioned, Theorem \ref{FIRSTMAINTHEOREM}' is proved by computing the moments of $Z_n$;
similarly Theorem \ref{brownian} is proved by computing the mixed moments of the finite dimensional distributions of $\widetilde Z_{n}^{(\B)}(t)$. The standard tool for calculating moments of this form is Rogers' mean value formula \cite{rogers1};  however, the assumption $\lim_{n\to\infty}f(n)=\infty$ causes divergence problems. To get around these, we develop, in Section \ref{newrogerssection}, a new version of Rogers' formula suitable for calculating moments of functions that can be represented in the form
\begin{equation*}
\sum_{\vecm\in L\setminus\{\vec0\}}\rho(V_n|\vecm|^n)-\int_0^{\infty}\rho(x)\,dx
\end{equation*}
for suitable test functions $\rho$; in particular the formula can be applied to calculate moments of $\widetilde Z_n^{(\B)}(t)$. The proof of this formula is combinatorial in nature. % and is presented in Section \ref{newrogerssection}. 
Using the formula, in Section \ref{proofsection} we prove Theorems \ref{FIRSTMAINTHEOREM}' and \ref{Poissoncor},
and in Section \ref{brownianproofsec} we prove Theorem \ref{brownian}.
Finally in Section \ref{expgrowthsec} we prove Theorem \ref{OPTIMALFNGROWTHTHM}, 
by a careful analysis of the sizes of the various non-leading order terms appearing in the 
moment computation used to prove Theorem \ref{FIRSTMAINTHEOREM}'.

\subsection*{Acknowledgments}
We are grateful to Svante Janson for helpful discussions. 
%*** Also Alex Gorodnik for pointing out reference ***.

\section{A new version of Rogers' mean value formula}
\label{newrogerssection} 

%\subsection{Rogers' mean value formula}\label{RogersSection} 
To begin, we describe Rogers' original formula. Let $1\leq k\leq n-1$ and let $\rho:(\R^n)^k\to\R_{\geq0}$ be a non-negative Borel measurable function. In \cite{rogers1} Rogers proved the following remarkable identity: %formula:
\begin{align}\label{rogformula}
&\int_{X_n}\sum_{\vecm_1,\ldots,\vecm_k\in L\setminus\{\vec0\}}\rho(\vecm_1,\ldots,\vecm_k)\,d\mu_n(L)\\
&=%\sum_{(\nu;\mu)}
\sum_{q=1}^{\infty}\sum_{D}
\Big(\frac{e_1}{q}\cdots\frac{e_m}{q}\Big)^n\int_{\R^n}\cdots\int_{\R^n}\rho\Big(\sum_{i=1}^m\frac{d_{i1}}{q}\vecx_i,\ldots,\sum_{i=1}^m\frac{d_{ik}}{q}\vecx_i\Big)\,d\vecx_1\ldots d\vecx_m.\nonumber
\end{align}
Here the inner sum is over all integer matrices $D=(d_{ij})$ having size $m\times k$ for some $1\leq m\leq k$,
satisfying the following properties:
%$D$ has size $m\times k$ for some $1\leq m\leq k$;
No column of $D$ vanishes identically;
the entries of $D$ have greatest common divisor equal to 1;
and finally there exists a division $(\nu;\mu)=(\nu_1,\ldots,\nu_m;\mu_1,\ldots,\mu_{k-m})$ 
of the numbers $1,\ldots,k$ into two sequences $\nu_1,\ldots,\nu_m$ and $\mu_1,\ldots,\mu_{k-m}$, satisfying
\begin{align}\label{division}
& 1=\nu_1<\nu_2<\ldots<\nu_m\leq k,\nonumber\\
& 1<\mu_1<\mu_2<\ldots<\mu_{k-m}\leq k,\\
& \nu_i\neq\mu_j, \text{ if $1\leq i\leq m$, $1\leq j \leq k-m$},\nonumber
\end{align}
such that
\begin{align}\label{division2}
&d_{i\nu_j}=q\delta_{ij}, \hspace{10pt}i=1,\ldots,m,\,\,j=1,\ldots,m,\\
&d_{i\mu_j}=0,\hspace{10pt}\text{ if }\,\mu_j<\nu_i,\,\,i=1,\ldots,m,\,\,j=1,\ldots,k-m.\nonumber
\end{align}
We call these matrices \textit{$\langle k,q\rangle$-admissible}.\footnote{Note that the only $\langle k,1\rangle$-admissible
matrix with $m=k$ is the $k\times k$ identity matrix, and 
for $q>1$ there are no $\langle k,q\rangle$-admissible matrices with $m=k$.}
Finally $e_i=(\ve_i,q)$, $i=1,\ldots,m$, where $\ve_1,\ldots,\ve_m$ are the elementary divisors of the matrix $D$. We stress that the right-hand side of \eqref{rogformula} is a \emph{positive} infinite linear combination of integrals of $\rho$ over certain linear subspaces of $(\R^n)^k$.

\begin{remark}\label{CONVremark}
The formula \eqref{rogformula} should be understood as an equality in $\R_{\geq0}\cup\{+\infty\}$; if either side of \eqref{rogformula} is divergent, then so is the other side.
By Schmidt, \cite[Thm.\ 2]{schmidtconvergence}, 
if $\rho$ is bounded and of compact support then both sides of \eqref{rogformula}
are finite.
Hence, under this restriction we may remove the assumption that $\rho$ is non-negative,
i.e.\ the formula \eqref{rogformula} is in fact valid for any real-valued Borel measurable function $\rho$ 
on $(\R^n)^k$ which is bounded and of compact support, 
with both sides of \eqref{rogformula} being 
nicely absolutely convergent.
%absolutely convergent in the sense that the multiple sums and integrals applied to $|\rho|$ in place of $\rho$ are all finite.
\end{remark}

\begin{remark}
It follows from the conditions on the matrices $D$ and \cite[Thm.\ 14.5.1]{hua} that  we always have $e_1=1$, and hence $\big(\frac{e_1}{q}\cdots\frac{e_m}{q}\big)^n\leq q^{-n}$. 
\end{remark}

%\subsection{A new version of Rogers' mean value formula}
%Given a lattice $L\in X_n$ and a real number $x\geq0$, we recall that $R_{n,L}(x)=N_{n,L}(x)-x$, where $N_{n,L}(x)$ denotes the number of non-zero lattice points of $L$ in the closed ball of volume $x$ centered at the origin in $\R^n$. Recall also that $V_n$ denotes the volume of the unit ball in $\R^n$.

We now state our new version of Rogers' mean value formula.

\begin{thm}\label{newrogers}
Let $n>k>0$, and let $f_1,\ldots,f_k$ be real-valued Borel measurable functions on $\R^n$ which are 
bounded and of compact support.
Define the functions $F_1,\ldots,F_k$ on $X_n$ by
\begin{align}\label{Fj}
F_j(L):=%\int_0^{\infty}f_j(x)\,dR_{n,L}(x)=
\sum_{\vecm\in L\setminus\{\vec0\}}f_j(\vecm)-\int_{\R^n}f_j(\vecx)\,d\vecx.
\end{align}
Then
\begin{multline*}
\mathbb E\Big(\prod_{j=1}^kF_j(L)\Big)=\sum_{q=1}^{\infty}{\sum_{D}}'\Big(\frac{e_1}{q}\cdots\frac{e_m}{q}\Big)^n\\
\times\int_{\R^n}\cdots\int_{\R^n}f_1\Big(\sum_{i=1}^m\frac{d_{i1}}{q}\vecx_i\Big)\cdots f_k\Big(\sum_{i=1}^m\frac{d_{ik}}{q}\vecx_i\Big)\,d\vecx_1\ldots d\vecx_m,
\end{multline*}
where $'$ indicates that the inner sum is over all $\langle k,q\rangle$-admissible matrices $D$ with 
the property that there are at least two non-zero entries in each row.
\end{thm}

We note that in the simple case $k=1$, Theorem \ref{newrogers} states that  %Note in particular that Theorem \ref{newrogers} implies 
\begin{align*}
\E(F_1(L))=0.
\end{align*}
This is in fact an immediate consequence of Siegel's mean value formula; see \cite{siegel}.

\begin{proof}
Let $K=\{1,\dots,k\}$. Using \eqref{Fj} and \eqref{rogformula}, we get
\begin{align}\notag
&\mathbb E\Big(\prod_{j=1}^kF_j(L)\Big)
\\
&=\sum_{A\subset K}(-1)^{\#(K\setminus A)}\bigg(\prod_{j\in K\setminus A}\int_{\R^n}f_j(\vecx)\,d\vecx\bigg)\mathbb E\bigg(\prod_{j\in A}\Big(\sum_{\vecm_j\in L\setminus\{\vec0\}}f_j(\vecm_j)\Big)\bigg)\nonumber\\
\label{FIRSTAPPLROGERSnew}
&=\sum_{A\subset K}(-1)^{\#(K\setminus A)}\bigg(\prod_{j\in K\setminus A}\int_{\R^n}f_j(\vecx)\,d\vecx\bigg)\sum_{q=1}^{\infty}\sum_{D}\Big(\frac{e_1}{q}\cdots\frac{e_m}{q}\Big)^n\\
&\hspace{100pt}\times\int_{\R^n}\cdots\int_{\R^n}
%f_{j_1}\Big(V_n\Big|\sum_{i=1}^m\frac{d_{i1}}{q}\vecx_i\Big|^n\Big)\cdots f_{j_a}\Big(V_n\Big|\sum_{i=1}^m\frac{d_{ia}}{q}\vecx_i\Big|^n\Big)
\prod_{\ell=1}^a f_{j_\ell}\Big(\sum_{i=1}^m\frac{d_{i\ell}}{q}\vecx_i\Big)
\,d\vecx_1\ldots d\vecx_m,\nonumber
\end{align}
where $A$ runs through all subsets of $K$,
we write $a=\#A$ and $A=\{j_1,\ldots,j_a\}$ with $j_1<\cdots<j_a$,
and the inner sum is taken over all $\langle a,q\rangle$-admissible matrices $D$.
As usual $m=m(D)$ denotes the number of rows of $D$.
Note that all multiple sums and integrals appearing in \eqref{FIRSTAPPLROGERSnew} are
absolutely convergent, because of our assumptions on $f_1,\ldots,f_k$; cf.\ Remark \ref{CONVremark}.

Given any $A=\{j_1,\ldots,j_a\}$, $q$ and $D$ appearing in the sum,
%so that $D$ is a $(\nu;\mu)_a$-admissible matrix of size $m\times a$,
we set $m':=m+k-a$ and write $K\setminus A=\{j_1',\ldots,j_{k-a}'\}$ with $j_1'<\cdots<j_{k-a}'$.
We then let $D'=D'(A,D)=(d_{ij}')$ be the $m'\times k$ matrix
which has $d_{i,j_\ell}'=d_{i,\ell}$ for $\langle i,\ell\rangle\in\{1,\ldots,m\}\times\{1,\ldots,a\}$,
$d_{m+\ell,j'_\ell}=q$ for $\ell=1,\ldots,k-a$,
and all other entries %$d_{ij}'$ 
equal to zero. Note that the matrix $D'$ is typically not $\langle k,q\rangle$-admissible.
Let $\ve_1',\ldots,\ve'_{m'}$ be the elementary divisors of $D'$ and set $e_j'=(\ve_j',q)$.
Then $e_1'\cdots e_{m'}'=q^{\#(K\setminus A)}e_1\cdots e_m$
(cf., e.g., \cite[Lemma 1]{rogers1}),
and so $\frac{e_1}{q}\cdots\frac{e_m}{q}=\frac{e_1'}{q}\cdots\frac{e'_{m'}}{q}$.
We may now rewrite each product of integrals in the right-hand side of \eqref{FIRSTAPPLROGERSnew} in terms of the matrices 
$D'=D'(A,D)=(d'_{ij})$: 
\begin{multline}\label{SECONDAPPLROGERSnew}
\mathbb E\Big(\prod_{j=1}^kF_j(L)\Big)=\sum_{A\subset K}(-1)^{\#(K\setminus A)}\sum_{q=1}^{\infty}\sum_{D}\Big(\frac{e_1'}{q}\cdots\frac{e'_{m'}}{q}\Big)^n\\
\times\int_{\R^n}\cdots\int_{\R^n}
\prod_{\ell=1}^k f_\ell\Big(\sum_{i=1}^{m'}\frac{d'_{i\ell}}{q}\vecx_i\Big)
\,d\vecx_1\ldots d\vecx_{m'}.
\end{multline} 
Note that any matrix $D'=D'(A,D)$ appearing in this sum can be brought, by a unique row permutation,
into a $\langle k,q\rangle$-admissible matrix
(this is easily seen by considering the admissibility conditions column by column, starting from the left).
Conversely, given any $\langle k,q\rangle$-admissible matrix $D'$,
let $S(D')$ %\subset\{\nu'_1,\ldots,\nu'_{m'}\}$ 
be the set of indices of those columns of $D'$ 
which have the property that the column has a unique non-zero entry and this entry is also the only non-zero entry in its row.
Then the matrix $D'$ is attained as a row permutation of $D'(A,D)$ 
for exactly $2^{\#S(D')}$ pairs $\langle A,D\rangle$ appearing in the above sum,
namely exactly once for each $B\subset S(D')$.
Hence
\begin{multline}
\mathbb E\Big(\prod_{j=1}^kF_j(L)\Big)=\sum_{q=1}^{\infty}\sum_{D'}\Big(\sum_{B\subset S(D')}(-1)^{\#B}\Big)\Big(\frac{e_1'}{q}\cdots\frac{e_{m'}'}{q}\Big)^n\\
\hspace{12pt}\times\int_{\R^n}\cdots\int_{\R^n}\prod_{\ell=1}^k f_\ell\Big(\sum_{i=1}^{m'}\frac{d'_{i\ell}}{q}\vecx_i\Big)\,d\vecx_1\ldots d\vecx_{m'},
\end{multline}
where now the sum over $D'$ is taken over all $\langle k,q\rangle$-admissible matrices.
But here $\sum_{B\subset S(D')}(-1)^{\#B}$ equals $1$ if $S(D')=\emptyset$ and equals $0$ otherwise.
Hence we obtain the formula stated in the theorem.
\end{proof}

\begin{remark}
Note the close connection between the formula in Theorem \ref{newrogers} and the formula in \cite[Prop.\ 7.1]{epstein2}.
\end{remark}

\begin{remark}
Clearly the family of functions $f_j$ admitted in Theorem~\ref{newrogers} can be extended by approximation arguments. 
However the present family is more than sufficient for our purposes in this paper.
%However, since the present family %of bounded, compactly supported and Borel measurable functions 
%is more than sufficient for our purposes, we do not pursue any such generalizations here.
\end{remark}

\begin{remark}
The formula in Theorem \ref{newrogers} is useful in the study of the Epstein zeta function $E_n(L,s)$. Recall from \cite[Sect.\ 4]{epstein2} that, for  $L\in X_n$ and $s\in\C\setminus\{0,\frac n2\}$, we have 
\begin{align}\label{FN}
\pi^{-s}\Gamma(s)E_n(L,s)=H_n(L,s)+H_n(L^*,\tfrac n2-s),
\end{align}
where $L^*$ is the dual lattice of $L$,
\begin{align*}
H_n(L,s):=-\frac{1}{\sfrac{n}{2}-s}+{\sum_{\vecm\in L\setminus\{\vec0\}}}G\big(s,\pi |\vecm|^2\big),
\end{align*}
and
\begin{align*}
G(s,x):=\int_1^{\infty}t^{s-1}e^{-xt}\,dt, \qquad \Re x>0.
\end{align*}
The connection between $E_n(L,s)$ and the present discussion comes from the relation 
\begin{align*}
H_n(L,s)=\int_0^{\infty}G\left(s,\pi\left(V_n^{-1}x\right)^{2/n}\right)\,dR_{n,L}(x), \qquad 0<s<\tfrac n2
\end{align*}
(cf.\ \cite[Eq.\ (4.7)]{epstein2}). It follows that Theorem \ref{newrogers} can be used to calculate (truncated) moments of $H_n(L,s)$. Furthermore, since $H_n(L,s)$ dominates $H_n(L^*,\frac n2-s)$ in the interval $(\frac14+\ve)n<s<\frac n2$ ($\ve>0$ fixed) for most lattices $L\in X_n$ when $n$ is large enough, we also find that the (truncated) moments of $H_n(L,s)$ are of apparent interest in the study of $E_n(L,s)$ in the limit as $n\to\infty$. We do not pursue this further here since we plan to give a detailed account of this topic elsewhere.
\end{remark}

We close this section by giving a generalization of Theorem \ref{newrogers} 
which seems potentially useful, although it will not be used in the present paper.
%It can be established with essentially the same method of proof. 
%, and which seems to be potentially useful, although we will not need it in the present paper.
%This more general result seems to be of intrinsic interest, but we will now use it in the present paper.
%This generalization will not be used in the present paper but we believe it may be useful in future investigations.

\begin{thm}\label{newrogersgen}
Let $k,\ell>0$ and $n>k\ell$. Let $g_j:(\R^n)^k\to\R$, $1\leq j\leq \ell$, be 
Borel measurable functions which are bounded and of compact support.
Consider the related functions $G_j: X_n\to\R$ defined by
\begin{align*}
G_j(L):=\sum_{\vecm_1,\ldots,\vecm_k\in L\setminus\{\vec0\}}g_j(\vecm_1,\ldots,\vecm_k)-\mathbb E\bigg(\sum_{\vecm_1,\ldots,\vecm_k\in L\setminus\{\vec0\}}g_j(\vecm_1,\ldots,\vecm_k)\bigg).
\end{align*}  
%For each $(\nu;\mu)_{k\ell}$-admissible matrix $D$, we let $T(D)\subset\{\nu_1,\ldots,\nu_m\}$ denote the set of all column indices of $D$ for which the (unique) non-zero entry in that column is also the only non-zero entry in its row. 
Then
\begin{align*}
&\mathbb E\Big(\prod_{j=1}^{\ell}G_j(L)\Big)=\sum_{q=1}^{\infty}{\sum_{D}}^{*}\Big(\frac{e_1}{q}\cdots\frac{e_m}{q}\Big)^n
\\
&\times
\int_{\R^n}\cdots\int_{\R^n}
\prod_{j=0}^{\ell-1} g_{j+1}\Bigl(
\sum_{i=1}^{m}\frac{d_{i,jk+1}}{q}\vecx_i,\sum_{i=1}^{m}\frac{d_{i,jk+2}}{q}\vecx_i,
\cdots,\sum_{i=1}^{m}\frac{d_{i,jk+k}}{q}\vecx_i\Bigr)
\,d\vecx_1\ldots d\vecx_m,
%\int_{\R^n}\cdots\int_{\R^n}g_1\Big(\sum_{i=1}^m\frac{d_{i1}}{q}\vecx_i,\ldots,\sum_{i=1}^m\frac{d_{ik}}{q}\vecx_i\Big)\\%\cdots
%\cdots g_{\ell}\Big(\sum_{i=1}^m\frac{d_{i(k(\ell-1)+1)}}{q}\vecx_i,\ldots,\sum_{i=1}^m\frac{d_{i(k\ell)}}{q}\vecx_i\Big)\,d\vecx_1\ldots d\vecx_m,
\end{align*}
where $^{*}$ indicates that the inner sum is over all $\langle k\ell,q\rangle$-admissible matrices $D$ with the property
that there do not exist any $j\in\{0,\ldots,\ell-1\}$ and $1\leq i_1\leq i_2\leq m$ such that the submatrix
at rows $i_1,i_1+1,\ldots,i_2$ and columns $jk+1,jk+2,\ldots,jk+k$ of $D$ 
is a multiple of a $\langle k,q'\rangle$-admissible
matrix for some $q'\mid q$, and all the remaining entries of these rows and columns of $D$ are zero.
%that the inner sum is over all $(\nu;\mu)_{k\ell}$-admissible matrices $D$ with no multiple of a $(\nu',\mu')_k$-admissible matrix $M$ sitting as a connected block of $D$ in such a way that the remaining entries of the corresponding rows and columns of $D$ are zero and the columns of $M$ correspond to a fixed function $g_j$.
\end{thm}

\begin{proof}[Outline of proof]
%\textbf{Outline of proof.}
Mimicking the beginning of the proof of Theorem \ref{newrogers},
in particular expanding $\mathbb E\big(\prod_{j=1}^{\ell}G_j(L)\big)$ as much as possible using \eqref{rogformula},
we obtain the formula
\begin{align}\notag
\mathbb E\Big(\prod_{j=1}^{\ell}G_j(L)\Big)=
\sum_{A\subset\{1,\ldots,\ell\}}(-1)^{\ell-\# A}
\sum_{\{q_j\}}\sum_{\{D_j\}}\sum_{q=1}^\infty\sum_D
\biggl(\frac{e_1'}{q'}\cdots\frac{e_{m'}'}{q'}\biggr)^n
\hspace{50pt}
\\\label{newrogersgenPF1}
\times\int_{\R^n}\cdots\int_{\R^n}
\prod_{j=0}^{\ell-1} g_{j+1}\Bigl(
\sum_{i=1}^{m'}\frac{d_{i,jk+1}'}{q'}\vecx_i,\cdots,\sum_{i=1}^{m'}\frac{d_{i,jk+k}'}{q'}\vecx_i\Bigr)
\,d\vecx_1\cdots d\vecx_{m'},
\end{align}
where the notation is as follows.
As before, $a=\#A$ and $A=\{j_1,\ldots,j_a\}$ with $j_1<\cdots<j_a$.
In the sums,
$\{q_j\}$ and $\{D_j\}$
are short-hands for
$\{q_j\}_{j\in A^c}$ and $\{D_j\}_{j\in A^c}$,
where $A^c$ %:=\{1,\ldots,\ell\}\setminus A$;
is the complement of $A$ in $\{1,\ldots,\ell\}$;
and $\{q_j\}$ runs through all $(\ell-a)$-tuples of positive integers 
while $\{D_j\}$ runs through all $(\ell-a)$-tuples of matrices such that
$D_j$ is $\langle k,q_j\rangle$-admissible for each $j\in A^c$.
In the innermost sum, $D$ runs through all $\langle ak,q\rangle$-admissible matrices.
For any $A,\{q_j\},\{D_j\},q,D$ appearing in the multiple sum
we let $q'$ be the least common multiple of $q$ and all the $q_j$'s,
and set $m'=m+\sum_{j\in A^c}m_j$, where $m$ is the number of rows of $D$ and 
$m_j$ is the number of rows of $D_j$.
Writing also $D_j=(d_{uv}^{(j)})$, $D=(d_{uv})$
and $\overline m_j:=m+\sum_{\substack{j'\in A^c\\ j'<j}}m_{j'}$,
we define $D'=D'(A,\{q_j\},\{D_j\},q,D)=(d_{ij}')$ to be the $m'\times k\ell$ matrix which has
$d'_{i,(j_u-1)k+v}={\displaystyle \frac{q'}q d_{i,(u-1)k+v}}$ for all 
$i\in\{1,\ldots,m\}$, $u\in\{1,\ldots,a\}$, $v\in\{1,\ldots,k\}$, 
and $d_{\overline m_j+i,(j-1)k+v}={\displaystyle \frac{q'}{q_j}d_{iv}^{(j)}}$ for all
$j\in A^c$, $v\in\{1,\ldots,k\}$, $i\in\{1,\ldots,m_j\}$,
and all other entries equal to zero.
Finally $e_j'=(\ve_j',q)$, where $\ve_1',\ldots,\ve_{m'}'$ are the elementary divisors of $D'$.
This completes the description of the notation in \eqref{newrogersgenPF1}.

One notes that each matrix $D'$ which appears above can be brought, by a unique row permutation,
into a $\langle k\ell,q'\rangle$-admissible matrix.
The rest of the proof follows closely the proof of Theorem \ref{newrogers}.
%is very similar to the end of the proof of Theorem \ref{newrogers}.
%with $S(D')$ now being defined as the set of those $j\in\{0,\ldots,\ell-1\}$ for which a condition as in the 
%\hfill$\square$
\end{proof}

\section{Proofs of Theorem \ref{FIRSTMAINTHEOREM}' and Theorem \ref{Poissoncor}}\label{proofsection}

Our first goal is to prove Theorem \ref{FIRSTMAINTHEOREM}' (and thus also Theorem \ref{FIRSTMAINTHEOREM}).
%as a special case).
Let $f,S_n$ and $Z_n$ be as in the statement of the theorem.
Thus $f:\Z^+\to \R^+$ is a function satisfying $\lim_{n\to\infty}f(n)=\infty$
and $f(n)=O_\epsilon(e^{\epsilon n})$ for every $\ve>0$;
for each $n$, $S_n$ is a Borel measurable subset of $\R^n$ which has volume $f(n)$ and which is symmetric about the origin
(viz., $-S_n=S_n$), and finally
\begin{align}\label{ZNDEFrep}
Z_{n}:=\frac{\#(L\cap S_n\setminus\{\bn\})-f(n)}{\sqrt{2f(n)}},
\end{align}
with $L$ picked at random in $(X_n,\mu_n)$.
It follows from Siegel's formula \cite{siegel} that for each $n\geq2$ we have
$\mathbb E\bigl(\#(L\cap S_n\setminus\{\bn\})\bigr)=f(n)$ and thus $\E(Z_n)=0$.
Using Theorem~\ref{newrogers}, we now determine the limits as $n\to\infty$ of the higher moments of $Z_n$.

\begin{prop}\label{MOMENTPROP}
For any fixed $k\in\Z^+$,
\begin{align*}
\lim_{n\to\infty}\mathbb E\big(Z_n^{\:k}\big)=
\begin{cases}
0 & \text{ if $k$ is odd,}\\
(k-1)!! & \text{ if $k$ is even.}
\end{cases}
\end{align*}
\end{prop}

\begin{proof}
Let $\chi_{n}$ be the characteristic function of $S_n$.
For any $n>k$, Theorem \ref{newrogers} gives
\begin{align} \label{Zmoments}
\mathbb E\big(Z_n^{\:k}\big)
&=\frac1{(2f(n))^{k/2}}\,\mathbb E\bigg(\bigg(
\sum_{\vecm\in L\setminus\{\vec0\}}\chi_n(\vecm)-\int_{\R^n}\chi_n(\vecx)\,d\vecx\bigg)^{\hspace{-3pt} k}\,\,\biggr)
\\\notag
&=\frac1{(2f(n))^{k/2}}\sum_{q=1}^{\infty}{\sum_{D}}'\Big(\frac{e_1}{q}\cdots\frac{e_m}{q}\Big)^n
\int_{\R^n}\cdots\int_{\R^n}\prod_{j=1}^k\chi_n\Big(\sum_{i=1}^m\frac{d_{ij}}{q}\vecx_i\Big)\,d\vecx_1\ldots d\vecx_m.
\end{align}
We let
\begin{align}\label{Mkndef}
M_{k,n}:={\sum_{D}}''\int_{\R^n}\cdots\int_{\R^n}\prod_{j=1}^k\chi_n\Big(\sum_{i=1}^m d_{ij}\vecx_i\Big)\,d\vecx_1\ldots d\vecx_m,
\end{align}
where the sum is taken over all $\langle k,1\rangle$-admissible matrices $D$ 
having entries $d_{ij}\in\{0,\pm1\}$,
with at least two non-zero entries in each row and exactly one non-zero entry in each column. 
Let $R_{k,n}$ be the sum of all the terms in \eqref{Zmoments} that are not accounted for in $M_{k,n}$,
so that 
\begin{align}\label{Zmoments2}
\mathbb E\big(Z_n^{\:k}\big)&=(2f(n))^{-k/2}\big(M_{k,n}+R_{k,n}\big).
\end{align}

Now, let $\mathcal P'(k)$ denote the set of partitions of $\{1,\ldots,k\}$ containing no singleton sets. 
% Can't say this here; since m is used already in [Zmoments]:
%As usual we write $m=m(D)$ for the number of rows of $D$.
Using $S_n=-S_n$ and $\vol(S_n)=f(n)$,
and then \cite[Lemma 3]{poisson}, we have
\begin{align}\label{mainterm}
M_{k,n}={\sum_{D}}''f(n)^m=\sum_{P\in\mathcal P'(k)}2^{k-\#P}f(n)^{\#P}.
\end{align}
It remains to bound %understand 
the term $R_{k,n}$ in \eqref{Zmoments2}. 
The summation condition in $\sum'_D$ %(cf.\ Theorem \ref{newrogers})
implies that all matrices $D$ appearing in $R_{k,n}$ have at most $k-1$ rows. Hence, an easy modification of the arguments in \cite[Sect.\ 9]{rogers2} and \cite[Sect.\ 4]{rogers3} (see also \cite[Sect.\ ~3]{poisson}) gives that,
for $n$ sufficiently large, %for $n\gg_k1$,
\begin{align}\label{remainder}
0\leq R_{k,n}\ll \Big(\frac34\Big)^{n/2}f(n)^{k-1},
\end{align}
where the implied constant depends on $k$ but not on $n$. If $k$ is odd, then we may assume that $k\geq3$ and in this situation we have $\#P\leq(k-1)/2$ for every $P\in\mathcal P'(k)$. 
Recall that we are assuming $f(n)=O_\ve(e^{\ve n})$.
Hence it follows from \eqref{Zmoments2}, \eqref{mainterm} and \eqref{remainder} that, %for fixed $f$ and odd $k\geq3$, we have 
for any odd $k\geq3$,
\begin{align*}
\lim_{n\to\infty}\mathbb E\big(Z_n^{\:k}\big)=0.
\end{align*}
On the other hand, if $k$ is even, %and $f$ is fixed, 
then  \eqref{Zmoments2}, \eqref{mainterm} and \eqref{remainder} imply that 
\begin{align*}
\lim_{n\to\infty}\mathbb E\big(Z_n^{\:k}\big)=\#\left\{P\in\mathcal P'(k): \#B=2, \forall B\in P\right\}=(k-1)!!.
\end{align*}
This completes the proof of the proposition.
\end{proof}

\begin{remark}\label{VARIANCEremark}
Note that the variance of $Z_n$ can be controlled for a much larger class of functions $f$. Indeed,
for \textit{any} $f:\Z^+\to \R^+$ satisfying $\lim_{n\to\infty}f(n)=\infty$, we have 
\begin{align*}
\text{Var}(Z_{n})=1+O\big(\big(\tfrac34\big)^{n/2}\big)\qquad
\text{as }\: n\to\infty.
\end{align*}
Cf.\ \eqref{Zmoments2} and \eqref{remainder} and note that $k-1=k/2=1$ for $k=2$.
\end{remark}

\begin{proof}[Proof of Theorem \ref{FIRSTMAINTHEOREM}'] 
The desired convergence follows immediately from Proposition \ref{MOMENTPROP}.
\end{proof}

\begin{proof}[Proof of Theorem \ref{Poissoncor}]
Let $\ve>0$ be given. It follows from Theorem \ref{FIRSTMAINTHEOREM} that
there exist $x_0>0$ and $n_0\in\Z^+$ such that
for all $n\geq n_0$, $x\in[x_0,f(n)]$ and $r\in\R$,
\begin{align}\label{Poissoncorpf4}
\biggl|\Prob\biggl(\frac{N_{n,L}(x)-x}{\sqrt{2x}}\leq r\biggr)-\frac1{\sqrt{2\pi}}\int_{-\infty}^r e^{-t^2/2}\,dt\biggr|
<\frac{\ve}2.
\end{align}
(Indeed, otherwise there is a sequence of positive integers $n_1<n_2<\cdots$ and positive numbers
$x_1,x_2,\ldots$ with $x_j\leq f(n_j)$ and $\lim_{j\to\infty} x_j=\infty$,
such that for each $j$, \eqref{Poissoncorpf4} fails for $n=n_j$, $x=x_j$ and some $r=r_j\in\R$.
We then obtain a contradiction against Theorem \ref{FIRSTMAINTHEOREM} applied to the function
$f_1:\Z^+\to\R^+$ given by $f_1(n_j)=x_j$ and, say, $f_1(n)=f(n)$ for $n\notin\{n_1,n_2,\ldots\}$.)
% Recalling also Shiryaev \cite[Ch.\ III.1, Probl 5]{shiryaev}.
Using also the fact that $\frac{2\scrN(x)-x}{\sqrt{2x}}$ tends in distribution to $N(0,1)$,
and taking $r=\frac{2N-x}{\sqrt{2x}}$,
it follows that after possibly increasing $x_0$, we have
\begin{align}\label{Poissoncorpf1}
\Bigl|\Prob(N_{n,L}(x)\leq 2N)-\Prob(\scrN(x)\leq N)\Bigr|<\ve
\end{align}
for all $n\geq n_0$, $x\in[x_0,f(n)]$, $N\geq0$.
On the other hand it follows from Theorem \ref{POISSONTHEOREM} (or \cite[Thm.\ 3]{rogers3})
that, after possibly increasing $n_0$, 
\eqref{Poissoncorpf1} also holds for all $n\geq n_0$, $x\in[0,x_0]$, $N\geq0$.
%This completes the proof.

Hence we have proved that \eqref{PoissoncorRES} holds uniformly with respect to all $N\geq0$ and $0\leq x\leq f(n)$.
The extension to the remaining case, i.e.\ $x>f(n)$ and $N\leq f(n)$, is now straightforward:
Applying what we have already proved to the function $n\mapsto 4f(n)$, 
it follows that the convergence in
\eqref{PoissoncorRES} holds uniformly with respect to all $N\geq0$ and $0\leq x\leq 4f(n)$;
thus it only remains to consider the case when $x>4f(n)$ and $N\leq f(n)$.
However, for such $x$ and $N$, we have
\begin{align}\label{PoissoncorPF1}
\Prob_{\mu_n}(N_{n,L}(x)\leq 2N)\leq\Prob_{\mu_n}(N_{n,L}(4f(n))\leq 2f(n))
\end{align}
and 
\begin{align}\label{PoissoncorPF2}
\Prob(\scrN(x)\leq N)\leq\Prob(\scrN(4f(n))\leq f(n)).
\end{align}
Here the right-hand side of \eqref{PoissoncorPF2} tends to zero as $n\to\infty$,
and so by the convergence already established also the 
right-hand side of \eqref{PoissoncorPF1} tends to zero.
Hence also the left-hand sides of \eqref{PoissoncorPF1} and \eqref{PoissoncorPF2}
tend to zero as $n\to\infty$,
uniformly over all $x>4f(n)$ and $N\leq f(n)$.
This concludes the proof.
\end{proof}

\section{Joint distribution for families of subsets, and proof of Theorem \ref{brownian}} 
\label{brownianproofsec}

Our main goal in this section is to prove Theorem \ref{brownian}.
As a first step, we generalize Proposition~\ref{MOMENTPROP} and Theorem \ref{FIRSTMAINTHEOREM}' to finite
families of disjoint subsets of $\R^n$.
Specifically, let us again fix a function $f:\Z^+\to\R^+$
satisfying $\lim_{n\to\infty}f(n)=\infty$
and $f(n)=O_\epsilon(e^{\epsilon n})$ for every $\ve>0$.
Fix a positive integer $r$ and positive real numbers $c_1,\ldots,c_r$.
For each $n$, let $S_{1,n},\ldots,S_{r,n}$ %$S_{j,n}$ for $j=1,\ldots,r$ 
be Borel measurable subsets of $\R^n$ satisfying
$\vol(S_{j,n})=c_jf(n)$, $-S_{j,n}=S_{j,n}$, and $S_{j,n}\cap S_{j',n}=\emptyset$ for all $j\neq j'$.
In analogy with \eqref{ZNDEFrep} we set
\begin{align}\label{ZJNDEF}
Z_{j,n}:=\frac{\#(L\cap S_{j,n}\setminus\{\bn\})-c_jf(n)}{\sqrt{2f(n)}},
\end{align}
with $L$ picked at random in $(X_n,\mu_n)$.
\begin{prop}\label{MOMENTPROP2}
In this situation, for any fixed $\veck=(k_1,\ldots,k_r)\in\Z_{\geq0}^r$,
\begin{multline*}
\lim_{n\to\infty}\mathbb E\Big(Z_{1,n}^{\:k_1}\cdots Z_{r,n}^{\:k_r}\Big)
=\begin{cases}
\prod_{j=1}^r\bigl(c_j^{k_j/2}(k_j-1)!!\bigr) & \text{ if $k_1,\ldots,k_r$ are all even,}\\
0 & \text{ otherwise.}
\end{cases}
\end{multline*}
\end{prop}

\begin{proof}
Set $\widehat k=k_1+\cdots+k_r$. 
Let $\chi_{j,n}$ be the characteristic function of $S_{j,n}$.
For any $n>\widehat k$, Theorem \ref{newrogers} gives
\begin{multline}\label{ZTILDECALCULATION}
\mathbb E\Big(Z_{1,n}^{\:k_1}\cdots Z_{r,n}^{\:k_r}\Big)=(2f(n))^{-\widehat k/2}\sum_{q=1}^{\infty}{\sum_{D}}'\Big(\frac{e_1}{q}\cdots\frac{e_m}{q}\Big)^n\\
\hspace{12pt}\times\int_{\R^n}\cdots\int_{\R^n}\prod_{j=1}^{r}\prod_{\ell_j=1}^{k_j}\chi_{j,n}\Big(\sum_{i=1}^m\frac{d_{i,k_1+\cdots+k_{j-1}+\ell_j}}{q}\vecx_i\Big)\,d\vecx_1\ldots d\vecx_m,
\end{multline}
where the sum over $D=(d_{ij})$ runs through all $\langle\widehat k,q\rangle$-admissible matrices
with the property that there are at least two non-zero entries in each row.
As in the proof of Proposition \ref{MOMENTPROP}, we divide the right-hand side into two parts as
\begin{align*}
\mathbb E\Big(Z_{1,n}^{\:k_1}\cdots Z_{r,n}^{\:k_r}\Big)=(2f(n))^{-\widehat k/2}\big(\widetilde M_{\veck,n}+\widetilde R_{\veck,n}\big),
\end{align*}
where
\begin{align}\label{TILDEM}
\widetilde M_{\veck,n}:={\sum_{D}}''\int_{\R^n}\cdots\int_{\R^n}\prod_{j=1}^{r}\prod_{\ell_j=1}^{k_j}\chi_{j,n}\Big(\sum_{i=1}^md_{i,k_1+\cdots+k_{j-1}+\ell_j}\,\vecx_i\Big)\,d\vecx_1\ldots d\vecx_m,
\end{align}
the sum being taken over all $\langle\widehat k,1\rangle$-admissible matrices having entries $d_{ij}\in\{0,\pm1\}$,
with at least two non-zero entries in each row and exactly one non-zero entry in each column. 
%and $\widetilde R_{\veck,n}$ is the sum of all the remaining terms in \eqref{ZTILDECALCULATION}. 
Using the assumption that $S_{1,n},\ldots,S_{r,n}$ are pairwise disjoint it follows that the terms in the right-hand side of \eqref{TILDEM} are zero unless, for each $i\in\{1,\ldots,m\}$, there is some $j\in\{1,\ldots,r\}$ such that the $i$th row of $D$ has all its non-zero elements in columns corresponding to the fixed function $\chi_{j,n}$. The rest of the proof follows closely that of Proposition \ref{MOMENTPROP}. 
\end{proof}

Note that Proposition \ref{MOMENTPROP2} immediately implies the following theorem,
generalizing Theorem \ref{FIRSTMAINTHEOREM}'.
\begin{thm}\label{independentgaussians}
Fix $r\in\Z^+$, $c_1,\ldots,c_r>0$, and a function
$f:\Z^+\to\R^+$ satisfying $\lim_{n\to\infty}f(n)=\infty$
and $f(n)=O_\epsilon(e^{\epsilon n})$ for every $\ve>0$.
For each $n$, let $S_{1,n},\ldots,S_{r,n}$ be Borel measurable subsets of $\R^n$ 
which are pairwise disjoint, and which satisfy $\vol(S_{j,n})=c_jf(n)$ and $-S_{j,n}=S_{j,n}$.
Set
\begin{align}\notag
Z_{j,n}:=\frac{\#(L\cap S_{j,n}\setminus\{\bn\})-c_jf(n)}{\sqrt{2f(n)}},
\end{align}
with $L$ picked at random in $(X_n,\mu_n)$.
Then
\begin{align*}
\bigl(Z_{1,n},\ldots,Z_{r,n}\bigr)\xrightarrow[]{\textup{ d }}\big(N(0,c_1),N(0,c_2),\ldots,N(0,c_r)\big)\qquad\text{as }\: n\to\infty,
\end{align*}
where the random vector in the right-hand side has independent coordinates.
\end{thm}

We are now in position to complete the proof of Theorem \ref{brownian}.

\begin{proof}[Proof of Theorem \ref{brownian}]
To simplify notation, in this proof we write $\widetilde Z_n(t):=\widetilde Z_n^{(\B)}(t)$.
Given any fixed numbers $0<t_1<t_2<\cdots<t_r\leq1$, by applying Theorem \ref{independentgaussians} 
with $S_{1,n},\ldots,S_{r,n}$ as the annuli
\begin{align*}
S_{j,n}=\biggl\{\vecx\in\R^n\col\Bigl(\frac{t_{j-1}f(n)}{V_n}\Bigr)^{1/n}<|\vecx|
\leq\Bigl(\frac{t_{j}f(n)}{V_n}\Bigr)^{1/n}\biggr\},
\qquad j=1,\ldots,r
\end{align*}
(with $t_0:=0$),
we conclude that the random vector
\begin{align*}
\Bigl(\widetilde Z_{n}(t_1),\widetilde Z_{n}(t_2)-\widetilde Z_{n}(t_1),\ldots,
\widetilde Z_{n}(t_r)-\widetilde Z_{n}(t_{r-1})\Bigr)
\end{align*}
tends in distribution to
\begin{align*}
%\xrightarrow[]{\textup{ d }}
\big(N(0,t_1),N(0,t_2-t_1),\ldots,N(0,t_r-t_{r-1})\big)
%\qquad\text{as }\: n\to\infty,
\end{align*}
as $n\to\infty$.
Note also that $\widetilde Z_{n}(0)=0$ by definition.
We have thus proved that the convergence in Theorem \ref{brownian} holds on the level of finite dimensional distributions,
and it now only remains to establish the tightness of the sequence $P_n$ of probability measures on $\mathcal D[0,1]$. 

By \cite[Thm.\ 13.5 and (13.14)]{billingsley2} (applied with $F(t)=C\sqrt t$ and $\beta=1$),
it suffices to prove that there exist $\alpha>\frac12$ and $N\in\N$ such that 
\begin{align}\label{moment condition}
\mathbb E\Big(\big(\widetilde Z_{n}(s)-\widetilde Z_{n}(r)\big)^2\big(\widetilde Z_{n}(t)-\widetilde Z_{n}(s)\big)^2\Big)\ll \big(\sqrt t-\sqrt r\big)^{2\alpha},
\end{align}
uniformly over all $0\leq r\leq s\leq t\leq1$ and $n\geq N$. We begin by noting that Proposition \ref{MOMENTPROP2} implies that %we have the asymptotic relation 
\begin{equation*}
\lim_{n\to\infty}\mathbb E\Big(\big(\widetilde Z_{n}(s)-\widetilde Z_{n}(r)\big)^{2}\big(\widetilde Z_{n}(t)-\widetilde Z_{n}(s)\big)^{2}\Big)=(t-s)(s-r)\leq(t-r)^2.
\end{equation*} 
Hence, using also the fact that
$t-r\leq2(\sqrt t-\sqrt r)$ for all $0\leq r\leq t\leq 1$,
we see that in the limit of large dimension $n$, \eqref{moment condition} holds with $\alpha=1$.
In order to get a more uniform statement, note that 
by naively modifying Rogers' arguments in \cite[Sect.\ 9]{rogers2} and \cite[Sect.\ 4]{rogers3} as in the proofs of Propositions \ref{MOMENTPROP} and \ref{MOMENTPROP2}, we have
\begin{multline}\label{remainder2}
\mathbb E\Big(\big(\widetilde Z_{n}(s)-\widetilde Z_{n}(r)\big)^{2}\big(\widetilde Z_{n}(t)-\widetilde Z_{n}(s)\big)^{2}\Big)\\
\ll(t-r)^2+\max\left(2^{-n}(t-r)f(n)^{-1},\Big(\frac34\Big)^{n/2}(t-r)^2,\Big(\frac34\Big)^{n/2}(t-r)^3f(n)\right)
%
% Write out the following instead??? Well, the point here is to be ``pedagogical'' regarding the derivation of the bound...
%
%\ll2^{-n}(t-r)f(n)^{-1}+(t-r)^2+\Big(\frac34\Big)^{n/2}(t-r)^3f(n)
\end{multline}
for all $n\geq6$, where the implied constant is absolute.

The bound \eqref{remainder2} is close but not quite sufficient for our purposes;
the problematic term is $2^{-n}(t-r)f(n)^{-1}$.
This term arises as a bound on the collected contribution of all 
$\langle 4,q\rangle$-admissible matrices $D$ with $m=1$ (and $q$ arbitrary) %, that is, matrices with only one row,
in the expression that is obtained by applying Theorem \ref{newrogers} to the left-hand side of \eqref{remainder2}
(cf.\ \eqref{ZTILDECALCULATION}).
Recall that $m=1$ means that $D$ has only one row.
In order to improve the bound,
note that any such matrix $D=(q,d_1,d_2,d_3)$ gives a contribution
\begin{align}\label{SPECIAL INTEGRAL}
\frac{1}{4q^nf(n)^2}\int_{\R^n}\chi_1\big(V_n|\vecx|^n\big)\chi_1\Big(V_n\Big|\frac{d_{1}}{q}\vecx\Big|^n\Big)\chi_2\Big(V_n\Big|\frac{d_{2}}{q}\vecx\Big|^n\Big)\chi_2\Big(V_n\Big|\frac{d_{3}}{q}\vecx\Big|^n\Big)\,d\vecx
\end{align}
to the left-hand side of \eqref{remainder2},
where $\chi_1$ and $\chi_2$ are the characteristic functions of the open intervals $(rf(n),sf(n))$ and $(sf(n),tf(n))$,
respectively.
Let us temporarily assume that $r>0$. Then, for the integral in \eqref{SPECIAL INTEGRAL} to be non-zero, we must have 
\begin{equation*}
1<\left|\frac{d_2}{q}\right|^n,\left|\frac{d_3}{q}\right|^n<\frac{t}{r}=1+\frac{t-r}{r}.
\end{equation*}
Hence, since $d_2$ and $d_3$ are integers, we conclude that a (crude)
necessary condition for \eqref{SPECIAL INTEGRAL} to be non-zero is
\begin{equation*}%\label{Q CONDITION}
q^n>\frac{r}{t-r}.
\end{equation*}
%Now, let us return to Rogers' estimates in \cite[Sect.\ 9]{rogers2}. Suppose, as we may, that $n\geq6$ and let us restrict our attention to admissible matrices $D$ with $m=1$. 
Let $Q$ be the smallest value of $q\in\Z^+$ satisfying this inequality. %\eqref{Q CONDITION}.
Then, for $n\geq6$, the estimate \cite[p.\ 246 (line 20)]{rogers2} with $\sum_{q=1}^\infty$ replaced by $\sum_{q=Q}^\infty$ gives
%\footnote{Note that if $Q=1$, then this bound is actually worse than Rogers' original bound. When $Q=2$ this bound is of the same% quality as Rogers' bound. In all other cases our bound improves on Rogers' bound.} 
\begin{align}\nonumber
\sum_{q=Q}^{\infty}{\sum_{\substack{D\\(m=1)\hspace{-5pt}}}}'\hspace{7pt}\frac{1}{4q^nf(n)^2}\int_{\R^n}\chi_1\big(V_n|\vecx|^n\big)\chi_1\Big(V_n\Big|\frac{d_{1}}{q}\vecx\Big|^n\Big)\chi_2\Big(V_n\Big|\frac{d_{2}}{q}\vecx\Big|^n\Big)\chi_2\Big(V_n\Big|\frac{d_{3}}{q}\vecx\Big|^n\Big)\,d\vecx\\\label{CAREFUL ROGERS}
\ll Q^{5-n}(t-r)f(n)^{-1}.%\ll \left(\frac{t-r}{r}\right)^{1-4/n}(t-r)f(n)^{-1}=
%<(t-r)^{2-5/n}r^{5/n-1}f(n)^{-1}.
\end{align}
Replacing the term %Using this bound to replace the problematic term 
$2^{-n}(t-r)f(n)^{-1}$ in \eqref{remainder2}
by the bound in \eqref{CAREFUL ROGERS} %this bound $Q^{5-n}(t-r)f(n)^{-1}$,
and using $Q\geq\max(1,(r/(t-r))^{1/n})$,
we obtain, allowing now the implied constant to depend on $f$:
\begin{align}\notag
\mathbb E\Big(\big(\widetilde Z_{n}(s)-\widetilde Z_{n}(r)\big)^{2}\big(\widetilde Z_{n}(t)-\widetilde Z_{n}(s)\big)^{2}\Big)
\hspace{90pt}
\\\label{brownianPF2}
\ll(t-r)^2+(t-r)\min\biggl(1,\Bigl(\frac{t-r}r\Bigr)^{1-\frac5n}\biggr)
\hspace{20pt}
\\\notag
\ll (t-r)\min\biggl(1,\Bigl(\frac{t-r}r\Bigr)^{1-\frac5n}\biggr).
\end{align}
This bound is also valid when $r=0$, with the convention that $\min(1,\cdots)$ then equals $1$.

Now fix the constant $\frac12<\alpha<1$ in an arbitrary manner,
%\begin{align*}
%\frac{1-5/n}{2-2\alpha-5/n}\leq\frac{\alpha}{2\alpha-1},
%\qquad\forall n\geq6,
%\end{align*}
and then take $N\geq6$ so large that $1-\alpha-\frac5N>0$.
We then claim that
\begin{align}\label{brownianPF1}
(t-r)\min\biggl(1,\Bigl(\frac{t-r}r\Bigr)^{1-\frac5n}\biggr)
\ll \big(\sqrt t-\sqrt r\big)^{2\alpha},
\end{align}
uniformly over all $n\geq N$ and $0\leq r\leq t\leq 1$.
Indeed, if $t\geq 2r$ then \eqref{brownianPF1} 
is clear from $(\sqrt t-\sqrt r)^{2\alpha}\asymp (\sqrt t)^{2\alpha}=t^\alpha$.
In the remaining case, i.e.\ when $0<r\leq t<2r$,
we have $\sqrt t-\sqrt r\asymp (t-r)/\sqrt r$
%and it follows that \eqref{brownianPF1} holds whenever either
%$t-r\geq r^{\alpha/(2\alpha-1)}$
%or $t-r\leq r^{(1-\alpha-\frac5n)/(2-2\alpha-\frac 5n)}$.
%However, one of these is always true,
%since $0<r\leq1$ and $\alpha/(2\alpha-1)>1>(1-\alpha-\frac5n)/(2-2\alpha-\frac 5n)$.
and \eqref{brownianPF1} is equivalent to
$t-r\ll r^{(1-\alpha-\frac5n)/(2-2\alpha-\frac 5n)}$,
which is true since $(1-\alpha-\frac5n)/(2-2\alpha-\frac 5n)<1$ and $t<2r$.
This completes the proof of \eqref{brownianPF1},
and in view of \eqref{brownianPF2} we thus obtain \eqref{moment condition},
completing the proof of Theorem \ref{brownian}.
\end{proof}

\section{Moment bounds for exponentially growing volumes} % $f(n)$}
\label{expgrowthsec}

Our goal in this section is to prove Theorem \ref{OPTIMALFNGROWTHTHM}.
Thus, for each $n$ we assume given a Borel subset $S_n$ of $\R^n$ satisfying $\vol(S_n)=f(n)$ and $S_n=-S_n$.
Throughout the section we let $\chi_n$ denote the characteristic function of $S_n$.
Our task is to go back to the proof of Proposition \ref{MOMENTPROP} and improve the
bound on $R_{k,n}$, i.e.\ the sum of those terms in \eqref{Zmoments} which 
%are not accounted for in $M_{k,n}$, i.e.\ terms coming
come from $\langle k,q\rangle$-admissible matrices $D$ with at least two non-zero entries in
each row and such that either $q\geq2$, 
or some column contains more than one non-zero entry,
or some entry has absolute value $|d_{ij}|\geq2$.
It will turn out that the dominating contribution to $R_{k,n}$ comes from $\langle k,1\rangle$-admissible matrices $D$
of the form
\begin{align}\label{WORSTMATRICES}
D=\begin{pmatrix}
1&0&\cdots&0&\pm1
\\
0&1&\cdots&0&\pm1
\\
\vdots&\vdots&\ddots&\vdots&\vdots
\\
0&0&\cdots&1&\pm1
\end{pmatrix}
\qquad(\text{thus }\: m=m(D)=k-1).
\end{align}

\subsection{Auxiliary lemmas}
\label{AUXLEMMASsec}

In our first lemma, 
%bound the contribution from a given $\langle k,q\rangle$-admissible matrix $D$
by repeated use of an integral inequality of Rogers, \cite[Theorem 1]{rogers4},
we bound
$\int_{\R^n}\cdots\int_{\R^n}\prod_{j=1}^k\chi_n\big(\sum_{i=1}^m \frac{d_{ij}}q\vecx_i\big)\,d\vecx_1\ldots d\vecx_m$
%the integral over $(\R^n)^m$ 
from above by a product of integrals of the following form:
\begin{align}\label{Jandef}
J_a^{(n)}[c_1,\ldots,c_a]:=\int_{(\R^n)^a}I\Bigl(|\vecx_i|<1\,(i=1,\ldots,a)%(\forall i)
, \: %\text{ and }
\Bigl|\sum_{i=1}^a c_i\vecx_i\Bigr|<1\Bigr)\,d\vecx_1\cdots d\vecx_a.
\end{align}
Here $n\geq a\geq1$ and $c_1,\ldots,c_a\in\R_{>0}$,
and $I(\cdot)$ is the indicator function.
We extend the definition to the case $a=0$ by setting $J_0^{(n)}[\:]:=1$ for all $n$.

Let $D$ be a $\langle k,q\rangle$-admissible matrix of size $m\times k$, having at least two non-zero entries in each row.
Set $r=k-m$, let $(\nu;\mu)=(\nu_1,\ldots,\nu_m;\mu_1,\ldots,\mu_r)$ be as in
Section~\ref{newrogerssection},  %\eqref{division}, \eqref{division2},
and let $\mu_1',\ldots,\mu_r'$ be an arbitrary permutation of $\mu_1,\ldots,\mu_r$.
For $j=1,\ldots,r$, we set 
\begin{align*}
\oA_j=\bigl\{i\in\{1,\ldots,m\}\col d_{i,\mu_j'}\neq0\bigr\};
\qquad
A_j=\oA_j\setminus(\cup_{\ell<j}\oA_\ell),
\quad\text{and}\quad
a_j=\# A_j.
\end{align*}
Since %In view of our assumption that 
$D$ has at least two non-zero entries in each row,
%we see that 
the sets $A_1,\ldots,A_r$ form a partition of $\{1,\ldots,m\}$, possibly with $A_j=\emptyset$ for some $j$'s.
Hence $\sum_{j=1}^{r} a_j=m$.
%Note that the sets $A_1,\ldots,A_r$ are pairwise disjoint. Furthermore, 
%we have $\cup_{j=1}^{r}A_j=\cup_{j=1}^{r}\oA_j=\{1,\ldots,m\}$ and 

\begin{lem}\label{BASICVBOUNDlem}
For $D$ as above,
\begin{align}\label{BASICVBOUNDlemres}
\int_{\R^n}\cdots\int_{\R^n}\prod_{j=1}^k\chi_n\Big(\sum_{i=1}^m \frac{d_{ij}}q\vecx_i\Big)\,d\vecx_1\ldots d\vecx_m
\leq V_n^{-m}f(n)^m %\biggl(\frac{f(n)}{V_n}\biggr)^m
\prod_{j=1}^rJ^{(n)}_{a_j}\bigl[\bigl(|d_{i,\mu_j'}|/q\bigr)_{i\in A_j}\bigr].
%\biggl[\biggl(\frac{d_{i,\mu_j'}}q\biggr)_{i\in A_j}\biggr].
\end{align}
\end{lem}

\begin{proof}
We express the left-hand side of \eqref{BASICVBOUNDlemres} as an iterated integral in the following way.
For each $j\in\{1,\ldots,r\}$ we write $\vecx^{(j)}:=(\vecx_i)_{i\in A_j}\in(\R^n)^{a_j}$
and $d\vecx^{(j)}:=\prod_{i\in A_j}d\vecx_i$.
(If $a_j=0$ then we understand $(\R^n)^0$ and $d\vecx^{(j)}$ to be the singleton set $\{\vec0\}$
with its unique probability measure.) %and the unique probability measure on this set.)
Let $F_r(\vecx^{(1)},\ldots,\vecx^{(r)})$ be the constant function $1$, and set,
iteratively for $j=r,r-1,\ldots,1$,
\begin{align}\notag
F_{j-1}(\vecx^{(1)},\ldots,\vecx^{(j-1)})
\hspace{270pt}
\\\label{ITERATEINT}
:=
\int_{(\R^n)^{a_j}}\biggl(\prod_{i\in A_j}\chi_n(\vecx_i)\biggr)
\chi_n\biggl(\sum_{i=1}^m \frac{d_{i,\mu_j'}}q\vecx_i\biggr)F_j(\vecx^{(1)},\ldots,\vecx^{(j)})\,d\vecx^{(j)}.
\end{align}
Then the left-hand side of \eqref{BASICVBOUNDlemres} equals $F_0$.
(The sum $\sum_{i=1}^m (d_{i,\mu_j'}/q)\vecx_i$ appearing in the right-hand side of \eqref{ITERATEINT} 
is well-defined since $d_{i,\mu_j'}=0$ for all $i\in\{1,\ldots,m\}\setminus(A_1\cup\cdots\cup A_j)$.)

Now let $B$ be the closed ball of volume $f(n)$ centered at the origin in $\R^n$, and let $\chi_B$ be its characteristic
function. Using \eqref{ITERATEINT} and \cite[Theorem 1]{rogers4}, we have
\begin{align*}
F_{j-1}(\vecx^{(1)},\ldots,\vecx^{(j-1)})
&\leq\bigl(\sup F_j\bigr)
\int_{(\R^n)^{a_j}}\biggl(\prod_{i\in A_j}\chi_n(\vecx_i)\biggr)
\chi_n\biggl(\sum_{i=1}^m \frac{d_{i,\mu_j'}}q\vecx_i\biggr)\,d\vecx^{(j)}
\\
&\leq\bigl(\sup F_j\bigr)\int_{(\R^n)^{a_j}}\biggl(\prod_{i\in A_j}\chi_B(\vecx_i)\biggr)
\chi_B\biggl(\sum_{i\in A_j} \frac{d_{i,\mu_j'}}q\vecx_i\biggr)\,d\vecx^{(j)},
\end{align*}
since $\chi_B$ is the spherical symmetrization both of $\chi_n$ and of 
$\vecy\mapsto\chi_n(\vecy+\vecz)$ for any fixed $\vecz\in\R^n$.
%\begin{align*}\vecy\mapsto\chi_n\biggl(\vecy+\sum_{i\notin A_j}\frac{d_{i,\mu_j'}}q\vecx_i\biggr),\end{align*}
%for any given $\vecx^{(1)},\ldots,\vecx^{(j-1)}$.
Hence, since $B$ has radius $V_n^{-1/n}f(n)^{1/n}$, we conclude
\begin{align*}
\sup F_{j-1} %(\vecx^{(1)},\ldots,\vecx^{(j-1)})
\leq V_n^{-a_j}f(n)^{a_j}J^{(n)}_{a_j}\bigl[\bigl(|d_{i,\mu_j'}|/q\bigr)_{i\in A_j}\bigr]\cdot \sup F_j.
\end{align*}
Using this bound %iteratively 
for $j=1,\ldots,r$,
together with $\sum_{j=1}^{r} a_j=m$, we obtain \eqref{BASICVBOUNDlemres}.
\end{proof}

\vspace{5pt}

We say that a function $F:(\R^n)^m\to\R$ ($1\leq m\leq n$) is $\OO(n)$-invariant if
$F(k\vecx_1,\ldots,k\vecx_m)=F(\vecx_1,\ldots,\vecx_m)$ for all $k\in \OO(n)$.
When this holds, %Given such an %function $F$, 
we define $\oF:(\R^m)^m\to\R$ through
$\oF(\vecx_1,\ldots,\vecx_m)=F(\iota(\vecx_1),\ldots,\iota(\vecx_m))$, where
$\iota$ is any fixed Euclidean isometry of $\R^m$ into $\R^n$.
Note that $\oF$ is independent of the choice of $\iota$.
Given any $\vecx_1,\ldots,\vecx_m\in\R^m$, we denote by 
$[\vecx_1,\ldots,\vecx_m]$ the volume of the parallelotope in $\R^m$ spanned by $\vecx_1,\ldots,\vecx_m$.
%(Equivalently, $[\vecx_1,\ldots,\vecx_m]$ is the absolute value of the determinant formed by the column vectors 
%$\vecx_1,\ldots,\vecx_m$.)
Finally, we write $\omega_n:=nV_n$ for the volume of the $(n-1)$-sphere.

\begin{lem}\label{CHANGEVARlem}
Let $1\leq m\leq n$ and let $F:(\R^n)^m\to\R_{\geq0}$ be a non-negative Borel measurable function which is 
$\OO(n)$-invariant.
Then
\begin{align}\notag
&\int_{(\R^n)^m}F(\vecx_1,\ldots,\vecx_m)\,d\vecx_1\cdots d\vecx_m
\\\label{BETTERASYMPTDISClem3res}
&=\frac{\prod_{j=n-m+1}^n\omega_j}{\prod_{j=1}^m\omega_j}
\int_{(\R^m)^m}\oF(\vecx_1,\ldots,\vecx_m)\,
[\vecx_1,\ldots,\vecx_m]^{n-m}\,d\vecx_1 \ldots d\vecx_m.
\end{align}
\end{lem}

\begin{proof}
Let $\vece_1,\ldots,\vece_n$ be the standard unit vectors in $\R^n$.
Passing to polar coordinates and then performing the same substitution as in \cite[p.\ 754]{angles},
%($\vecx_j=r_j\vecu_j$ with $r_j>0$ and $\vecu_j\in\S_1^{n-1}$, for $j=1,\ldots,m$)
the left-hand side of \eqref{BETTERASYMPTDISClem3res} becomes
\begin{align*}
\Bigl(\prod_{j=n-m+1}^n\omega_j\Bigr)
\int_{(\R_{>0})^m}\int_{(0,\pi)^{M}}F(\vecx_1,\ldots,\vecx_m)
\prod_{1\leq i<j\leq m}(\sin\phi_{i,j})^{n-i-1} \prod_{j=1}^m r_j^{n-1}\,d\vecphi\,d\vecr,
\end{align*}
where $M=\binom m2$,
$\vecr=(r_1,\ldots,r_m)$,  $\vecphi=(\phi_{i,j})_{1\leq i<j\leq m}$,
and $d\vecr$ and $d\vecphi$ denote Lebesgue measure on $\R^m$ and $\R^M$, respectively,
and 
\begin{align}\label{BETTERASYMPTDISClem1res}
\vecx_j=r_j\biggl(\sum_{1\leq i<j}\Bigl(\prod_{i'<i}\sin\phi_{i',j}\Bigr)(\cos\phi_{i,j})\vece_i
+\Bigl(\prod_{i'<j}\sin\phi_{i',j}\Bigr)\vece_j\biggr).
\end{align}
(In particular $\vecx_1=r_1\vece_1$.)
We view $\R^m$ as a subspace of $\R^n$ through $(x_1,\ldots,x_m)\mapsto(x_1,\ldots,x_m,0,\ldots,0)$.
Then all the $\vecx_j$ in \eqref{BETTERASYMPTDISClem1res} lie in $\R^m$.
The desired formula %\eqref{BETTERASYMPTDISClem3res} 
now follows by performing the same substitutions backwards, in $\R^m$ instead of in $\R^n$,
and using $(\prod_{j=1}^m r_j)\prod_{i<j}\sin\phi_{i,j}=[\vecx_1,\ldots,\vecx_m]$.
\end{proof}

\vspace{5pt}

Applying Lemma \ref{CHANGEVARlem} to the integral in \eqref{Jandef}, we see that the 
asymptotics of $J_a^{(n)}[c_1,\ldots,c_a]$ as $n\to\infty$ depends mainly on the quantity
%Now for any $m\geq1$ and $c_1,\ldots,c_m\in\R_{>0}$ we set
\begin{align}\notag
\V_a[c_1,\ldots,c_a]:=\sup\Bigl\{[\vecx_1,\ldots,\vecx_a]\col \vecx_1,\ldots,\vecx_a\in\R^a,\:
|\vecx_j|\leq1\,(\forall j),\,
\hspace{50pt}
\\\label{VMdef}
|c_1\vecx_1+\cdots+c_a\vecx_a|\leq1\Bigr\}.
\end{align}

\begin{lem}\label{BASICJBOUNDSlem}
For any $1\leq a\leq n$ and $c_1,\ldots,c_a>0$,
\begin{align}\label{BASICJBOUNDSlemres}
J_a^{(n)}[c_1,\ldots,c_a]\ll_a n^{a(a+3)/4} V_n^a\,\V_a[c_1,\ldots,c_a]^{n-a}.
\end{align}
On the other hand, for any fixed $c_1,\ldots,c_a\in\R_{>0}$ and $\V\in(0,\V_a[c_1,\ldots,c_a])$,
we have $\lim_{n\to\infty}\V^{-n}V_n^{-a}J_a^{(n)}[c_1,\ldots,c_a]=\infty$.
\end{lem}
\begin{proof}
Let $B$ be the open unit ball in $\R^a$ centered at the origin.
Then Lemma \ref{CHANGEVARlem} gives
\begin{align}\notag
J_a^{(n)}[c_1,\ldots,c_a]
=\frac{\prod_{j=n-a+1}^n\omega_j}{\prod_{j=1}^a\omega_j}
\int_{B^a}I\biggl(\biggl|\sum_{i=1}^a c_i\vecx_i\biggr|<1\biggr)\,[\vecx_1,\ldots,\vecx_a]^{n-a}\,d\vecx_1\cdots d\vecx_a
%\\\label{BASICJBOUNDSlempf1}
%&\leq\frac{\prod_{j=n-a+1}^n\omega_j}{\prod_{j=1}^a\omega_j}\int_{B^a}\V_a[c_1,\ldots,c_a]^{n-a}\,d\vecx_1\cdots d\vecx_a =
\\\label{BASICJBOUNDSlempf1}
\leq\frac{\prod_{j=n-a+1}^n\omega_j}{\prod_{j=1}^a\omega_j} V_a^a\V_a[c_1,\ldots,c_a]^{n-a}.
\end{align}
Furthermore, by Stirling's formula,
\begin{align}\label{BASICJBOUNDSlempf2}
\omega_j=jV_j=\frac{2\pi^{j/2}}{\Gamma(j/2)}\asymp_a n^{1+(n-j)/2}V_n
\end{align}
for all $j\in\{n-a+1,n-a+2,\ldots,n\}$.
These two bounds imply \eqref{BASICJBOUNDSlemres}.

Next, let $c_1,\ldots,c_a\in\R_{>0}$ and $\V$ be given as in the statement of the lemma.
%Fix $\V'$ such that $\V<\V'<\V_a[c_1,\ldots,c_a]$.
It is clear from \eqref{VMdef} that 
there exist non-empty open subsets $U_1,\ldots,U_a$ of $B$ such that
all $(\vecx_1,\ldots,\vecx_a)\in U_1\times\cdots\times U_a$ satisfy both
$|c_1\vecx_1+\cdots+c_a\vecx_a|<1$ and $[\vecx_1,\ldots,\vecx_a]>\V$.
Using the first equality in \eqref{BASICJBOUNDSlempf1}, it follows that
\begin{align*}
J_a^{(n)}[c_1,\ldots,c_a]\geq\frac{\prod_{j=n-a+1}^n\omega_j\prod_{j=1}^a\vol(U_j)}{\prod_{j=1}^a\omega_j}{\V}^{n-a}.
\end{align*}
Using this and \eqref{BASICJBOUNDSlempf2}, the second claim of the lemma follows.
\end{proof}

The next lemma gives a bound on the product $\frac{e_1}q\cdots\frac{e_m}q$ appearing in
\eqref{Zmoments}.
Recall that $e_i=(\ve_i,q)$, %for $i=1,\ldots,m$, 
where $\ve_1,\ldots,\ve_m$ are the elementary divisors of the matrix $D$.

\begin{lem}\label{BASICVBOUNDlem2}
For any $D$ as in Lemma \ref{BASICVBOUNDlem},
\begin{align}\label{BASICVBOUNDlem2res}
\frac{e_1}q\cdots\frac{e_m}q\leq\prod_{j=1}^{r}\frac{g_j}q,\qquad
\text{with }\: g_j=\gcd\bigl(\{q\}\cup\{d_{i,\mu_j'}\col i\in A_j\}\bigr).
\end{align}
\end{lem}
(Note that if $A_j=\emptyset$ then $g_j=q$, giving a factor $1$ 
%i.e.\ the corresponding factor 
in the product in \eqref{BASICVBOUNDlem2res}.)
\begin{proof}
By \cite[Lemma 1]{rogers1},
%the left hand side of \eqref{BASICVBOUNDlem2res} equals $q^{-m}N(D,q)$, 
%$e_1\cdots e_m=N(D,q)$, where
\begin{align*}
e_1\cdots e_m=N(D,q):=%\#\{\veca\in(\Z/q\Z)^m\col\veca D=0\text{ in }(\Z/q\Z)^k\}.
\#\biggl\{(x_1,\ldots,x_m)\in(\Z/q\Z)^m\col\sum_{i=1}^m d_{ij}x_i\equiv0\mod q\:\:(\forall j)\biggr\}.
\end{align*}
As a preliminary step, note that for any integers $c,d_1,\ldots,d_\ell$,
\begin{align}\label{BASICVBOUNDlem2pf1}
\#\biggl\{(x_1,\ldots,x_\ell)\in(\Z/q\Z)^\ell\col \sum_{j=1}^\ell d_jx_j\equiv c\mod q\biggr\}\leq 
q^{\ell-1}\gcd(q,d_1,\ldots,d_\ell).
\end{align}
Indeed, this is immediate when $q$ is a prime power, and the general case can be reduced to this case using
the Chinese Remainder Theorem.
We now set $\tA_0:=\emptyset$ and $\tA_j:=A_1\cup\cdots\cup A_j=\oA_1\cup\cdots\cup\oA_j$ for $j\geq1$.
For any $j\in\{1,\ldots,r\}$
and any given $(x_i)_{i\in \tA_{j-1}}$ in $(\Z/q\Z)^{\#\tA_{j-1}}$,
it follows from \eqref{BASICVBOUNDlem2pf1} that the number of
tuples $(x_i)_{i\in A_j}\in(\Z/q\Z)^{a_j}$ satisfying
$\sum_{i=1}^m d_{i,\mu_j'}x_i\equiv0\mod q$ is less than or equal to
$q^{a_j-1}g_j$. 
%Recall that we have $\oA_j=A_1\cup\cdots\cup A_j$ by construction; let us also set $\oA_0=\emptyset$.
%For any $j\in\{1,\ldots,r\}$
%and any given $(x_i)_{i\in \oA_{j-1}}$ in $(\Z/q\Z)^{\#\oA_{j-1}}$,
%it follows from \eqref{BASICVBOUNDlem2pf1} that the number of
%tuples $(x_i)_{i\in A_j}\in(\Z/q\Z)^{a_j}$ satisfying
%$\sum_{i=1}^m d_{i,\mu_j'}x_i\equiv0\mod q$ is less than or equal to
%$q^{a_j-1}g_j$.
% This is true also if $a_j=0$.
Using this fact for each $j=1,\ldots,r$, we obtain
\begin{align*}
N(D,q)\leq\prod_{j=1}^r\bigl(q^{a_j-1}g_j\bigr)
=q^m\prod_{j=1}^{r}\frac{g_j}q.
\end{align*}
This completes the proof of the lemma. 
\end{proof}

\subsection{Some basic properties of $\V_a[c_1,\ldots,c_a]$}

Recall that, for any integer $a\geq1$ and real numbers $c_1,\ldots,c_a>0$, %\in\R_{>0}$,
\begin{align}\notag
\V_a[c_1,\ldots,c_a]:=\sup\Bigl\{[\vecx_1,\ldots,\vecx_a]\col \vecx_1,\ldots,\vecx_a\in\R^a,\:
|\vecx_j|\leq1\,(\forall j),\,
\hspace{50pt}
\\\label{Vadef}
|c_1\vecx_1+\cdots+c_a\vecx_a|\leq1\Bigr\},
\end{align}
where $[\vecx_1,\ldots,\vecx_a]$ denotes the volume of the parallelotope in $\R^a$ spanned by $\vecx_1,\ldots,\vecx_a$.
Note that $0<\V_a[c_1,\ldots,c_a]\leq1$, and $\V_a[c_1,\ldots,c_a]$ is invariant under any permutation of 
$c_1,\ldots,c_a$.

\begin{lem}\label{VMlem2}
$\V_a[c_1,\ldots,c_a]= c_1^{-1}\V_a[c_1^{-1},c_1^{-1}c_2,\ldots,c_1^{-1}c_a]$,
for any $c_1,\ldots,c_a>0$.
\end{lem}
\begin{proof}
Set $\vecd:=c_1\vecx_1+\ldots+c_a\vecx_a$ and note that
$\vecx_1=c_1^{-1}(\vecd-\sum_{j=2}^ac_j\vecx_j)$
and $[\vecx_1,\ldots,\vecx_a]=c_1^{-1}[\vecd,\vecx_2,\ldots,\vecx_a]$.
Hence the lemma follows by substituting $\vecx_1=\vecd^{(old)}$ and
$\vecx_j=-\vecx_j^{(old)}$ ($j\geq2$) in the definition of 
$\V_a[c_1^{-1},c_1^{-1}c_2,\ldots,c_1^{-1}c_a]$.
\end{proof}

\begin{lem}\label{VMlem1}
If $c_1^2+\cdots+c_a^2\leq1$, then $\V_a[c_1,\ldots,c_a]=1$.
Furthermore, we have $\V_a[c_1,\ldots,c_a]\leq c_\ell^{-1}$ for each $\ell\in\{1,\ldots,a\}$,
and if $c_\ell^2\geq1+\sum_{j\neq\ell}c_j^2$ then $\V_a[c_1,\ldots,c_a]=c_\ell^{-1}$.
\end{lem}
\begin{proof}
The first statement is clear by taking $\vecx_1,\ldots,\vecx_a$ to be an ON-basis in the definition of
$\V_a[c_1,\ldots,c_a]$.
The remaining statements follow from the first statement of the lemma,
combined with the general bound $\V_a[c_1,\ldots,c_a]\leq1$,
Lemma \ref{VMlem2}, and the invariance of $\V_a[c_1,\ldots,c_a]$ under permutations of $c_1,\ldots,c_a$.
\end{proof}

\begin{remark}\label{CASEa1}
For $a=1$ we have $\V_1[c]=\min(1,c^{-1})$.
%\begin{align*}\V_1[c]=\begin{cases}1&\text{if }\: c\leq1,\\ c^{-1}&\text{if }\: c\geq1.\end{cases}\end{align*}
This is clear directly from the definition, or from Lemma \ref{VMlem1}.
\end{remark}

\begin{lem}\label{VaCONTLEM}
For any $c_1,\ldots,c_a>0$ and $c_1',\ldots,c_a'>0$,
\begin{align*}
\V_a[c_1',\ldots,c_a']\geq\Bigl(1+\sum_{j=1}^a|c_j-c_j'|\Bigr)^{-a}\V_a[c_1,\ldots,c_a].
\end{align*}
In particular $\V_a$ is a continuous function on $(\R_{>0})^a$.
\end{lem}
\begin{proof}
Set $\delta=(1+\sum_{j=1}^a|c_j-c_j'|)^{-1}\leq1$.
Let $\vecx_1,\ldots,\vecx_a$ be vectors which achieve the supremum in \eqref{Vadef}.
Then
\begin{align*}
|c_1'\vecx_1+\cdots+c_a'\vecx_a|\leq|c_1\vecx_1+\cdots+c_a\vecx_a|+\sum_{j=1}^a|c_j-c_j'|\leq\delta^{-1}.
\end{align*}
Hence the vectors $\delta\vecx_1,\ldots,\delta\vecx_a$ are admissible in the supremum defining $\V_a[c_1',\ldots,c_a']$,
so that $\V_a[c_1',\ldots,c_a']\geq[\delta\vecx_1,\ldots,\delta\vecx_a]=\delta^a\V_a[c_1,\ldots,c_a]$.
\end{proof}

The following technical lemma gives the key input both to %the proof of 
a monotonicity property of $\V_a$ which we will need (Lemma \ref{VaMONOTONICITYlem}),
and to the explicit determination of $\V_a[c_1,\ldots,c_a]$ in the case $c_1=\cdots=c_a$
(Lemma \ref{WVlem}).
\begin{lem}\label{texistsLEM}
Assume $c_1,\ldots,c_a>0$, $c_1^2+\cdots+c_a^2>1$
and $c_j^2<1+\sum_{\ell\neq j}c_\ell^2$ for each $j$.
% Thus $a\geq2$.
Let $\vecx_1,\ldots,\vecx_a$ be vectors %in $\R^a$ 
which achieve the supremum in \eqref{Vadef}.
Let $\vecd:=c_1\vecx_1+\ldots+c_a\vecx_a$, and for each $j\in\{1,\ldots,a\}$,
let $\delta_j$ be the length of the orthogonal projection of $\vecd$ onto the subspace 
$U_j=\Span\{\vecx_\ell\col \ell\in\{1,\ldots,a\}\setminus\{j\}\}$.
Then, for each $j\in\{1,\ldots,a\}$,
\\[3pt]
(i) there is $\ve>0$ such that $c_j'\in(c_j-\ve,c_j)\Rightarrow\V_a[c_1,\ldots,c_j',\ldots,c_a]>\V_a[c_1,\ldots,c_a]$;
%for all $c_j'\in(c_j-\ve,c_j)$;
\\[3pt]
(ii) $\delta_j^2+c_j^2>1$, and the number
%\\
%(iii) Also, the number 
$\delta_j^2-(\delta_j-\delta_j^3)(\delta_j^2+c_j^2-1)^{-1/2}$ 
%${\displaystyle \delta_j^2-\frac{\delta_j-\delta_j^3}{(\delta_j^2+c_j^2-1)^{1/2}}}$ 
is independent of $j$.
%\\
%(iii) $(c_j^2+1)\delta_j^2<1$ for each $j\in\{1,\ldots,a\}$;
%\\
%(iv) for each $j\in\{1,\ldots,a\}$,
%$\V_a[c_1,\ldots,c_j',\ldots,c_a]\geq\V_a[c_1,\ldots,c_j,\ldots,c_a]$
%holds for all $c_j'$ in some interval $(c_j-\ve,c_j)$, $\ve>0$.
\end{lem}

\begin{proof}
For each $j$, %$j\in\{1,\ldots,a\}$
$\vecx_j\notin U_j$ since $\V_a[c_1,\ldots,c_a]>0$;
we let $\vece_j$ be the unique unit vector in $\R^a$ which is orthogonal to $U_j$
and satisfies $\vecx_j\cdot\vece_j>0$.
Let $\vecd_j$ %and $\tx_j$ 
be the orthogonal projection of $\vecd$ %and $\vecx_j$ 
onto $U_j$; thus $\delta_j=|\vecd_j|$.

Let us fix $j$ temporarily, and set $\vecy=\vecd-c_j\vecx_j\in U_j$ and $y=|\vecy|$.
%Note that %by our choice of $\vecx_1,\ldots,\vecx_a$,
%Our assumption on $\vecx_1,\ldots,\vecx_a$ 
The optimality property of $\vecx_1,\ldots,\vecx_a$ 
implies in particular that among all $\vecx_j'\in\R^a$ satisfying $|\vecx_j'|\leq1$ and
$|c_j\vecx_j'+\vecy|\leq1$,
the vector $\vecx_j'=\vecx_j$ has maximal distance from $U_j$.
%(But there may exist other such vectors $\vecx_j'$ which have the same distance from $U_j$.)
By a straightforward analysis one deduces from this fact (and $\vecx_j\cdot\vece_j>0$)
that
\begin{align}\label{texistsLEMpfnew7add}
\vecx_j=-\alpha\vecy+\beta\vece_j,
\end{align}
with
\begin{align}\label{texistsLEMpfnew7}
\begin{cases}
\alpha=0\:\text{ and }\:\beta=1&\text{if }\: y^2\leq 1-c_j^2,
\\
\alpha=\beta=c_j^{-1} %\:\text{ and }\:\beta=c_j^{-1}
&\text{if }\: y^2\leq c_j^2-1,
\\[5pt]
\alpha=(2c_{j}y^2)^{-1}(y^2+c_{j}^2-1)
%{\displaystyle\frac{c_{j}^2+y^2-1}{2c_{j}y^2}}
\:\text{ and }\:\beta=\sqrt{1-(\alpha y)^2}
\hspace{10pt}
&\text{if }\: y^2>|c_j^2-1|.
\end{cases}
\end{align}

Let us first assume that $y^2<1-c_j^2$.
Then $\vecx_j=\vece_j$ by \eqref{texistsLEMpfnew7}
and $|\vecd|=|\vecy+c_j\vecx_j|=(y^2+c_j^2)^{1/2}<1$,
and so the optimality property of $\vecx_1,\ldots,\vecx_a$ 
%our assumption on $\vecx_1,\ldots,\vecx_a$
%so the assumption that $\vecx_1,\ldots,\vecx_a$ achieve the supremum in \eqref{Vadef}
forces $\{\vecx_\ell\col\ell\neq j\}$ to be an orthonormal basis of $U_j$.
Hence $\sum_{\ell\neq j}c_\ell^2=y^2<1-c_j^2$, which contradicts our assumption that $c_1^2+\cdots+c_a^2>1$.
This shows that $y^2<1-c_j^2$ cannot hold.

Similarly, $y^2<c_j^2-1$ is impossible.
Indeed, $[\vecx_1,\ldots,\vecx_a]=c_j^{-1}[\vecd,\vecx_1,\ldots,\widehat\vecx_j,\ldots,\vecx_a]$
(where $\widehat\vecx_j$ denotes omission of $\vecx_j$ in the list),
and hence
the optimality property of $\vecx_1,\ldots,\vecx_a$ can be rephrased as saying that
%is equivalent to
the $a$ vectors 
$\vecd,\vecx_1,\ldots,\widehat\vecx_j,\ldots,\vecx_a$
maximize $[\vecd,\vecx_1,\ldots,\widehat\vecx_j,\ldots,\vecx_a]$ subject to
$|\vecd|\leq1$, $|\vecx_\ell|\leq1$ (all $\ell\neq j$) and $|\vecd-\sum_{\ell\neq j}c_\ell\vecx_\ell|\leq c_j$.
Assume now $y^2<c_j^2-1$.
Then \eqref{texistsLEMpfnew7} gives 
$\vecd=\vecy+c_j\vecx_j=\vece_j$ and 
$|\vecx_j|^2=(y/c_j)^2+(1/c_j)^2<1$,
and so the optimality property just noted forces $\{\vecx_\ell\col\ell\neq j\}$
to again be an orthonormal basis of $U_j$.
Therefore $\sum_{\ell\neq j} c_\ell^2=y^2<c_j^2-1$,
contradicting our assumption that $c_j^2<1+\sum_{\ell\neq j} c_\ell^2$.
%Hence $y^2<c_j^2-1$ cannot hold.

In conclusion, $y^2\geq|c_j^2-1|$ must hold.
Let us also assume $y>0$.
Then one verifies that the formulas for $\alpha$ and $\beta$ in the third line of \eqref{texistsLEMpfnew7} 
hold true (viz.,  they remain valid even when $y^2=|c_j^2-1|$).
These formulas imply $|\vecx_j|=|\vecd|=1$.
Using $\vecd_j=(1-c_j\alpha)\vecy$ and the formula for $\alpha$, we obtain
$\delta_j=(y^2+1-c_j^2)/(2y)$ and $0\leq\delta_j\leq y$.
Solving for $y$ gives $\delta_j^2+c_j^2\geq1$
%Therefore $(y-\delta_j)^2=\delta_j^2+c_j^2-1$, $\delta_j^2+c_j^2\geq1$,
and
\begin{align}\label{texistsLEMpfnew9}
y=\delta_{j}+\tau_j,\quad\text{with }\:
\tau_j:=(\delta_{j}^2+c_{j}^2-1)^{1/2}.
\end{align}
Eliminating $\vecy$ from $\vecx_j=-\alpha\vecy+\beta\vece_j$ and
$\vecd=\vecy+c_j\vecx_j$ gives $(1-c_j\alpha)\vecx_j=-\alpha\vecd+\beta\vece_j$,
and here $1-c_j\alpha=\delta_j/y$.
Hence $c_j\delta_j\vecx_j=c_jy(\beta\vece_j-\alpha\vecd)$.
Using \eqref{texistsLEMpfnew9}, we obtain $c_j\alpha y=\tau_j$
and $c_j\beta=(1-\delta_j^2)^{1/2}$.
Therefore
\begin{align}\label{texistsLEMpfnew20}
c_j\delta_j\vecx_j=(\delta_j+\tau_j)(1-\delta_j^2)^{1/2}\vece_j-\tau_j\vecd.
\end{align}
We take note of two more facts. First:
\begin{align}\label{texistsLEMpfnew30}
\vecd\cdot\vece_j=(c_j\vecx_j+\vecy)\cdot\vece_j=c_j\beta=(1-\delta_j^2)^{1/2}>0.
\end{align}
Second:
\begin{align}\label{texistsLEMpfnew10}
\tau_j=0\Rightarrow\vecx_j=\vece_j.
\end{align}
Indeed, $\tau_j=0$ implies $y=\delta_j=(y^2+1-c_j^2)/(2y)$ by \eqref{texistsLEMpfnew9}; thus
$y^2=1-c_j^2$, giving $\vecx_j=\vece_j$. % by \eqref{texistsLEMpfnew7}.

In the remaining case $y=0$,
we have $c_j=1$ (since $y^2\geq|c_j^2-1|$)
and $\vecx_j=\vece_j$ (by \eqref{texistsLEMpfnew7}, \eqref{texistsLEMpfnew7add});
thus also $\vecd=%\vecy+c_j\vecx_j=
\vece_j$,
$\delta_j=\tau_j=0$, and all of \eqref{texistsLEMpfnew9}--\eqref{texistsLEMpfnew10} are still valid.

We now prove the first half of (ii), which asserts that in fact $\tau_j>0$ must hold for all $j$.
Assume $\tau_i=0$ for some $i$.
Then $\vecx_i=\vece_i$ by \eqref{texistsLEMpfnew10},
and now for every $j\neq i$ we have $\vece_j\cdot\vece_i=\vece_j\cdot\vecx_i=0$, since $\vecx_i\in U_j$.
Similarly $\vecx_j\cdot\vece_i=0$. %, since $\vecx_j\in U_i$.
Therefore $\tau_j\vecd\cdot\vece_i=0$, by \eqref{texistsLEMpfnew20};
but $\vecd\cdot\vece_i>0$ (cf.\ \eqref{texistsLEMpfnew30}); hence $\tau_j=0$.
It follows that $\tau_j=0$ and $\vecx_j=\vece_j$ for \textit{all} $j$;
hence $\vecx_1,\ldots,\vecx_a$ is an orthonormal basis of $\R^a$.
Then $1=|\vecd|^2=c_1^2+\cdots+c_a^2$, which contradicts one of our assumptions.
Hence indeed $\tau_j>0$ %must hold 
for all $j$.

Next, for any $i\neq j$ in $\{1,\ldots,a\}$, we compute
$c_ic_j\delta_i\delta_j\vecx_i\cdot\vecx_j$ in two different ways.
On the one hand, using \eqref{texistsLEMpfnew20} and $\vecx_i\cdot\vece_j=0$,  % true since $\vecx_j\in U_i$
we have
\begin{align}\notag
c_ic_j\delta_i\delta_j\vecx_i\cdot\vecx_j=c_i\delta_i\vecx_i\cdot(-\tau_j\vecd)
=\bigl((\delta_i+\tau_i)(1-\delta_i^2)^{1/2}\vece_i-\tau_i\vecd\bigr)\cdot(-\tau_j\vecd)
\\\label{texistsLEMpfnew31}
%=-\tau_j(\delta_i+\tau_i)(1-\delta_i^2)+\tau_i\tau_j,
=\tau_j(\tau_i\delta_i^2+\delta_i^3-\delta_i),
\end{align}
where in the last equality we used $|\vecd|=1$ and $\vece_i\cdot\vecd=(1-\delta_i^2)^{1/2}$.
On the other hand, by symmetry, the same formula holds with $i$ and $j$ interchanged. Thus
\begin{align}\label{texistsLEMpfnew21}
\tau_j(\tau_i\delta_i^2+\delta_i^3-\delta_i)
=\tau_i(\tau_j\delta_j^2+\delta_j^3-\delta_j).
%\tau_j(\delta_i+\tau_i)(1-\delta_i^2)=\tau_i(\delta_j+\tau_j)(1-\delta_j^2).
\end{align}
This holds for all $i\neq j$,
and dividing through with $\tau_i\tau_j$, 
% NO NEED TO SAY ``(recalling $\tau_i,\tau_j>0$)'' -- since this was so recently done, and is also in the statement of (ii)...
we have proved (ii).

Let $t$ be the number $\delta_j^2-\tau_j^{-1}(\delta_j-\delta_j^3)$, which is independent of $j$.
Let us first assume that $\delta_\ell=0$ for some $\ell$.
%Then $t=0$, i.e.\ $\tau_j\delta_j^2=\delta_j-\delta_j^3=0$ for all $j$,
%and a quick computation using $\tau_j=(\delta_j^2+c_j^2-1)^{1/2}$ shows that this implies
%$\delta_j=0$ or $\delta_j=(c_j^2+1)^{-1/2}$.
Then $t=0$, and also $\vecd\cdot\vece_\ell=1$ by \eqref{texistsLEMpfnew30},
and since $|\vecd|=1$ this forces $\vecd=\vece_\ell$.
For each $j\neq\ell$, we have $U_j\neq U_\ell$ and thus $\delta_j>0$. %by def of $\delta_j$!  i.e.\ $\delta_j=(c_j^2+1)^{-1/2}$.
%Hence since $\vece_j\neq\vece_\ell$ for all $j\neq\ell$
%(since $\vecx_j\cdot\vece_\ell=0$ but $\vecx_j\cdot\vece_j>0$),
%we must have $\delta_j=(c_j^2+1)^{-1/2}$ for all $j\neq\ell$.
For any $i\neq j$, the right-hand side of \eqref{texistsLEMpfnew31} vanishes, since $t=0$,
and if further $i,j\neq\ell$ then we may divide through with $\delta_i\delta_j$ to conclude that
$\vecx_i\cdot\vecx_j=0$.
% This is an empty assertion if a=2; however no problem; the following conclusion is still ok:
Hence $\{\vece_\ell\}\cup\{\vecx_j\col j\neq\ell\}$ is an orthonormal basis of $\R^a$.
Now, from $c_\ell\vecx_\ell=\vecd-\sum_{j\neq\ell}c_j\vecx_j=\vece_\ell-\sum_{j\neq\ell}c_j\vecx_j$
it follows that $c_\ell^2=1+\sum_{j\neq\ell}c_j^2$,
which contradicts our assumption that $c_\ell^2<1+\sum_{j\neq\ell}c_j^2$.
Hence we conclude that $\delta_j>0$ must hold for all $j$.
Expanding $1=|\vecd|^2=|\sum_j c_j\vecx_j|^2$ using \eqref{texistsLEMpfnew31}, we now obtain
\begin{align*}
1=\sum_{j=1}^a c_j^2+2\sum_{i<j}\frac{\tau_i\tau_j}{\delta_i\delta_j}t.
\end{align*}
In view of our assumption $\sum c_j^2>1$, this forces $t<0$.
Hence $\tau_j\delta_j<1-\delta_j^2$,
or equivalently $c_j^2>(\delta_j+\tau_j)^2-1$, for all $j$.

Now fix $j$ again, and write $\vecy=\vecd-c_j\vecx_j$ and $y=|\vecy|$ as before;
note that $y>0$ since $\tau_j>0$.
By \eqref{texistsLEMpfnew9},
$c_j^2>(\delta_j+\tau_j)^2-1$ means that $c_j^2>y^2-1$,
and this is easily seen to imply that there is some $\ve>0$ such that the function
$c\mapsto(y^2+c^2-1)/(2yc)$ is strictly increasing in the interval $c\in[c_j-\ve,c_j]$.
We have $y^2>1-c_j^2$ since $\tau_j>0$;
%We have proved that $\delta_j>0$ and $\tau_j>0$; this implies $y^2>|1-c_j^2|$.
%(Indeed, $y^2>c_j^2-1$ is clear from $(y^2+1-c_j^2)/(2y)=\delta_j>0$;
%we also know that $y^2\geq1-c_j^2$, and if $y^2=1-c_j^2$ then
%$c_j\leq1$ and $\delta_j=(1-c_j^2)^{1/2}$, giving $\tau_j=0$.)
hence, by shrinking $\ve$ if necessary, we may also assume that 
$(y^2+c^2-1)/(2yc)>0$ for all $c\in[c_j-\ve,c_j]$.
In particular, taking $\alpha,\beta$ as in %the third case of 
\eqref{texistsLEMpfnew7},
and setting, for any given $c_j'\in(c_j-\ve,c_j)$,
\begin{align*}
\alpha'=(2c_j'y^2)^{-1}(y^2+{c_j'}^2-1)\quad\text{and}\quad
\beta'=\sqrt{1-(\alpha'y)^2},
\end{align*}
we have $0<y\alpha'<y\alpha<1$, and hence $\beta'>\beta>0$.
Now set $\vecx'_j=-\alpha'\vecy+\beta'\vece_j$.
Then $|\vecx_j'|=1$ since $(\alpha'y)^2+{\beta'}^2=1$,
and $|\sum_{i\neq j}c_i\vecx_i+c_j'\vecx_j'|=|\vecy+c_j'\vecx_j'|=1$
since $(1-c_j'\alpha')^2y^2+{c_j'}^2{\beta'}^2=1$.
Hence
\begin{align*}
\V_a[c_1,\ldots,c_j',\ldots,c_a]\geq
[\vecx_1,\ldots,\vecx_j',\ldots,\vecx_a]=\frac{\beta'}{\beta}[\vecx_1,\ldots,\vecx_a]
&>[\vecx_1,\ldots,\vecx_a]
\\
&=\V_a[c_1,\ldots,c_a],
\end{align*}
which concludes the proof of (i).
\end{proof}

We next establish a monotonicity property of the function 
$\V_a[c_1,\ldots,c_a]$.
\begin{lem}\label{VaMONOTONICITYlem}
If $c_j\geq c_j'>0$ for $j=1,\ldots,a$, then $\V_a[c_1,\ldots,c_a]\leq\V_a[c_1',\ldots,c_a']$.
\end{lem}

\begin{proof}
It suffices to prove that for any fixed $c_2,\ldots,c_a>0$,
$\V_a[c_1,c_2,\ldots,c_a]$ is a decreasing function of $c_1>0$.
%(Indeed, the statement of the lemma then follows by iteration and invariance under permutations.)
Without loss of generality, we assume that $c_2\geq c_j$ for $j\geq3$.
Set 
\begin{align*}
\alpha=\max\Bigl(0,1-\sum_{j\geq2}c_j^2,\: c_2^2-1-\sum_{j\geq3}c_j^2\Bigr)^{1/2}
\quad\text{and}\quad
\beta=\Bigl(1+\sum_{j\geq2}c_j^2\Bigr)^{1/2}.
\end{align*}
Then for $\alpha<c_1<\beta$, Lemma \ref{texistsLEM} applies,
and part (i) of that lemma, together with the continuity of $\V_a$
(cf.\ Lemma \ref{VaCONTLEM}), implies that $c_1\mapsto\V_a[c_1,c_2,\ldots,c_a]$ is strictly decreasing for
$\alpha<c_1<\beta$.
In fact this is valid for $\alpha\leq c_1\leq\beta$, again by continuity.
Finally, Lemma \ref{VMlem1} implies that 
$c_1\mapsto\V_a[c_1,c_2,\ldots,c_a]$ is decreasing for
$0<c_1\leq\alpha$ and for $c_1\geq\beta$, and the proof is complete.
\end{proof}

The following lemma gives the exact value of $\V_a[c_1,\ldots,c_a]$ when $c_1=\cdots=c_a$.

\begin{lem}\label{WVlem}
For $a\geq1$ and $0<c\leq1$,
\begin{align*}
\tV_{a,c}:=\V_a[c,\ldots,c]=\begin{cases}
{\displaystyle \sqrt{\frac{c^{-2}(a^2-c^{-2})^{a-1}}{a^{a}(a-1)^{a-1}}}}
&\text{if }c>a^{-1/2},
\\[15pt]
1&\text{if }c\leq a^{-1/2}.
\end{cases}
% We have to exclude c=a^{-1/2} in the first case ONLY WHEN a=1!!
\end{align*}
\end{lem}

\begin{proof}
The case $c\leq a^{-1/2}$ follows from Lemma \ref{VMlem1};
hence we now assume $c>a^{-1/2}$ (and $a\geq2$).
Then Lemma \ref{texistsLEM} applies.
Let $\vecx_1,\ldots,\vecx_a$ and $\delta_1,\ldots,\delta_a$ be as in the statement of that lemma.
Set $\gamma:=1-c^2\in[0,1)$.
One verifies by differentiation that 
$\frac{\delta-\delta^3}{(\delta^2-\gamma)^{1/2}}$ is a strictly decreasing function of $\delta$ in the
interval $\sqrt\gamma<\delta\leq1$;
hence, a fortiori, $\delta^2-\frac{\delta-\delta^3}{(\delta^2-\gamma)^{1/2}}$ is strictly increasing in that interval.
Hence Lemma \ref{texistsLEM}\,(ii) implies $\delta_1=\cdots=\delta_a>\sqrt\gamma$.
Using this in the formula \eqref{texistsLEMpfnew31} (wherein $\tau_i=(\delta_i^2+c^2-1)^{1/2}$),
it follows that the scalar product $\vecx_i\cdot\vecx_j$ takes one and the same value for all choices of $i\neq j$.
Call this value $s$.
It was also seen in the proof of Lemma \ref{texistsLEM} that $|\vecx_j|=1$ for all $j$, and $|\sum_{j=1}^a c\vecx_j|=1$.
Squaring and expanding the last relation gives $c^2(a+a(a-1)s)=1$. We have thus proved
\begin{align*}
\vecx_i\cdot\vecx_j=s=\frac{c^{-2}-a}{a(a-1)},\qquad\text{for all }\: i\neq j.
\end{align*}
Hence
\begin{align*}
\V_a[c,\ldots,c]=[\vecx_1,\ldots,\vecx_a]=\sqrt{D_{a,s}},
\qquad\text{with }\:
D_{a,s}:=\left|\begin{matrix}1&s &\cdots&s
\\
s&1&\cdots&s
\\
\vdots& & \ddots & \vdots
\\
s&s&\cdots&1
\end{matrix}\right|.
\end{align*}
Subtracting $s$ times the first row from each of the other rows, we get
\begin{align*}
D_{a,s}=%\underbrace{
\left|\begin{matrix}1-s^2&s-s^2 &\cdots&s-s^2
\\
s-s^2&1-s^2&\cdots&s-s^2
\\
\vdots& & \ddots & \vdots
\\
s-s^2&s-s^2&\cdots&1-s^2
\end{matrix}\right|   %}_{\text{determinant of order $a-1$}}
=(1-s^2)^{a-1}D_{a-1,s/(1+s)},
\end{align*}
and from this one proves by induction that
$D_{a,s}=(as-s+1)(1-s)^{a-1}$.
This gives the formula stated in the lemma.
\end{proof}

%Lemma \ref{WVlem} and Lemma \ref{BASICJBOUNDSlem} together lead to information on the 
%order of magnitude of $J_a^{(n)}[c,\ldots,c]$ as $n\to\infty$.
%In the special case $c=1$, which corresponds to the dominating contribution in \eqref{Zmoments},
%we now give a more precise estimate. We set
The case $c=1$ will turn out to be of special importance, and we set
\begin{align}\label{TVdef}
\tV_a:=\tV_{a,1}=\sqrt{\frac{(a+1)^{a-1}}{a^a}}\quad (a\geq1);\qquad
\tV_0:=1.
\end{align}

\subsection{Proof of Theorem \ref{OPTIMALFNGROWTHTHM}}
For $k=2$, the statement of Theorem \ref{OPTIMALFNGROWTHTHM} follows from Remark \ref{VARIANCEremark}.
Hence, from now on we fix $k$ to be an integer $\geq3$.
We also fix $c$ and $f$ as in Theorem \ref{OPTIMALFNGROWTHTHM};
thus $0<c<c_k$, $\lim_{n\to\infty}f(n)=\infty$ and $f(n)=O(e^{c n})$.

The following lemma takes care of all except finitely many terms in \eqref{Zmoments};
it is proved using the same bounds as in Rogers, \cite[pp.\ 245--246]{rogers2},
%but in a slightly different notation.
which were also used in the proof of Proposition \ref{MOMENTPROP} above.
\begin{lem}\label{ALLOTHERMATRICESlem}
The total contribution to \eqref{Zmoments} from all $D$ which satisfy 
$\max\{|d_{ij}|\}\geq \tV_{k-1}^{-1}$ 
(the maximum being taken over all entries of $D$)
tends to zero as $n\to\infty$.
\end{lem}
\begin{remark}
If $k\leq10$ then $\tV_{k-1}^{-1}<2$, so that Lemma \ref{ALLOTHERMATRICESlem} 
in fact takes care of all $D$ except those which have $q=1$ and all entries $d_{ij}\in\{-1,0,1\}$.
%***Note that the argument in Lemma \ref{ALLOTHERMATRICESlem} is just ``ROgers' old''...
\end{remark}
\begin{proof}
%Mimicking \cite[p.\ 246, lines 5--7]{rogers2} we see that,
%assuming $n>m(k-m)+1$, the contribution to \eqref{Zmoments} from all $D$
%satisfying $q\geq v_k$ and $\max\{|d_{ij}|\}\leq q$ is $\ll_k v_k^{-n}f(n)^{m-k/2}$.
%
%The proof follows \cite[p.\ 246]{rogers2} closely.
%
We fix $m\in\{1,\ldots,k-1\}$, and consider the contribution from all $D$ as in the lemma with the further
requirement that %$m(D)=m$, i.e.\ 
$D$ is of size $m\times k$.
Set $\Delta:=\max\{|d_{ij}|\}$. Then, by \cite[Remark 1]{poisson} and \cite[(72)]{rogers2},
\begin{align*}
&%\frac1{(2f(n))^{k/2}}
\Big(\frac{e_1}{q}\cdots\frac{e_m}{q}\Big)^n
\int_{\R^n}\cdots\int_{\R^n}\prod_{j=1}^k\chi_n\Big(\sum_{i=1}^m \frac{d_{ij}}q\vecx_i\Big)\,d\vecx_1\ldots d\vecx_m
\leq f(n)^{m}\Delta^{-n}.    %&\leq q^{-n}\max(1,\fD)^{-n}f(n)^m,
\end{align*}
%Let $v_k\geq2$ be the smallest integer $\geq\tV_{k-1}^{-1}$.
Note that the number of $\langle k,q\rangle$-admissible matrices of size $m\times k$ and with 
given values of $q$ and $\Delta$,
is less than $\binom{k-1}{m-1}(3\Delta)^{m(k-m)}$, and there are \textit{no} such matrices with $\Delta<q$.
Hence, if we let $v_k$ be the smallest integer $\geq\tV_{k-1}^{-1}$ (thus $v_k\geq2$),
and assume that $n\geq m(k-m)+3$,
then the total contribution to \eqref{Zmoments} from all $D$ with $q\geq\tV_{k-1}^{-1}$ is
\begin{align*}
\leq  \binom{k-1}{m-1}f(n)^{m-k/2}\sum_{q=v_k}^\infty\sum_{\Delta\geq q}(3\Delta)^{m(k-m)}\Delta^{-n}
%<3^{m(k-m)}f(n)^{m-k/2}\sum_{\Delta=v_k}^\infty\Delta^{m(k-m)+1-n}
%\leq 3^{m(k-m)}f(n)^{m-k/2}v_k^{m(k-m)+3-n}\sum_{\Delta=v_k}^\infty\Delta^{-2}
%\\
\ll_k f(n)^{m-k/2}v_k^{-n}.
\end{align*}
Similarly, assuming $n\geq m(k-m)+2$, the total contribution to \eqref{Zmoments} from all $D$ satisfying $q<\tV_{k-1}^{-1}$ 
and $\Delta\geq\tV_{k-1}^{-1}$ (viz., $q<v_k$ and $\Delta\geq v_k$) is
\begin{align*}
\leq  \binom{k-1}{m-1}f(n)^{m-k/2}\sum_{q=1}^{v_k-1}\sum_{\Delta=v_k}^\infty(3\Delta)^{m(k-m)}\Delta^{-n}\ll_k f(n)^{m-k/2}v_k^{-n}.
\end{align*}
Finally, using $\lim_{n\to\infty}f(n)=\infty$ and $f(n)=O(e^{cn})$ with $0<c<c_k$, 
the desired convergence is seen to follow from the fact that
\begin{align*}
c\Bigl(m-\frac k2\Bigr)-\log v_k<c_k\Bigl(\frac k2-1\Bigr)+\log\tV_{k-1}=0,
\end{align*}
cf.\ \eqref{CKDEF} and \eqref{TVdef}.
\end{proof}

In the next three lemmas,
we let $D$ be any fixed $\langle k,q\rangle$-admissible matrix appearing in the sum in \eqref{Zmoments}.
(We could assume that $D$ does not satisfy the condition in Lemma \ref{ALLOTHERMATRICESlem}, but we won't need this.)
%In view of Lemma \ref{ALLOTHERMATRICESlem}, there now only remains finitely many matrices $D$ to consider from the
%sum in \eqref{Zmoments}, and so to complete the proof of (***the first part of) Theorem~\ref{OPTIMALFNGROWTHTHM}
%it suffices to prove that the contribution to \eqref{Zmoments}
%from each \textit{individual} term not accounted for in $M_{k,n}$ tends to zero as $n\to\infty$.
Let $m,r,(\mu_j')_{j=1}^r,(\oA_j)_{j=1}^r,(A_j)_{j=1}^r,(a_j)_{j=1}^r$ be as in Section \ref{AUXLEMMASsec}.

\begin{lem}\label{BASICBOUNDlem3}
%For any $D$ as in Lemma \ref{BASICVBOUNDlem}, 
%We have
If $n\geq\max(a_1,\ldots,a_r)$, then
\begin{align}\notag
\biggl(\frac{e_1}q\cdots\frac{e_m}q\biggr)^n
\int_{\R^n}\cdots\int_{\R^n}\prod_{j=1}^k\chi_n\Big(\sum_{i=1}^m \frac{d_{ij}}q\vecx_i\Big)\,d\vecx_1\ldots d\vecx_m
\hspace{100pt}
\\\label{BASICBOUNDlem3res}
\ll_m n^{m(m+3)/4}f(n)^m \biggl(\prod_{j=1}^r\tV_{a_j}\biggr)^n.
\end{align}
\end{lem}
\begin{proof}
By Lemmas \ref{BASICVBOUNDlem}, \ref{BASICJBOUNDSlem}, \ref{BASICVBOUNDlem2},
and using $\sum_{j=1}^r a_j=m$ (thus $\sum_{j=1}^r a_j^2\leq m^2$),
the left-hand side of \eqref{BASICBOUNDlem3res} is
\begin{align}\label{BASICBOUNDlem3pf1}
\ll_m n^{m(m+3)/4}f(n)^m 
\prod_{j=1}^r\Bigl(c_j^n\,\V_{a_j}\bigl[\bigl(|d_{i,\mu_j'}|/q\bigr)_{i\in A_j}\bigr]^{n-a_j}\Bigr),
%\quad\text{with }\: c_j:=\frac{g_j}q.
\end{align}
where $c_j:=q^{-1}\gcd\bigl(\{q\}\cup\{d_{i,\mu_j'}\col i\in A_j\}\bigr)$, % as in Lemma \ref{BASICVBOUNDlem2},
and we use the convention that $\V_0[\:]:=1$.
Using Lemma \ref{VaMONOTONICITYlem} and the fact that $|d_{i,\mu_j'}|\geq qc_j$ for all $i\in A_j$, we 
have $\V_{a_j}[(|d_{i,\mu_j'}|/q)_{i\in A_j}]\leq\tV_{a_j,c_j}$ for each $j$.
Note that $0<c_j\leq1$ by definition, and thus $\tV_{a_j,c_j}\geq\tV_{a_j}\gg_m1$.
%, which is bounded from below by a positive constant which only depends on $m$.
Also, inspecting the formula in Lemma \ref{WVlem},
one notes that for any fixed $a\geq1$, $\,c\,\tV_{a,c}$ is a strictly increasing function of $c\in(0,1]$;
on the other hand, for each $j$ with $a_j=0$ we have $c_j=1$ and $\tV_{a_j,c_j}=1$.
Using these facts, we see that for each $j\in\{1,\ldots,r\}$,
\begin{align}\label{BASICBOUNDlem3pf2}
c_j^n\,\V_{a_j}\bigl[\bigl(|d_{i,\mu_j'}|/q\bigr)_{i\in A_j}\bigr]^{n-a_j}
\ll_m \bigl(c_j\,\tV_{a_j,c_j}\bigr)^n\leq\tV_{a_j}^n.
\end{align}
Now \eqref{BASICBOUNDlem3res} follows from \eqref{BASICBOUNDlem3pf1} and \eqref{BASICBOUNDlem3pf2}.
\end{proof}

\begin{lem}\label{OPTIMALFNGROWTHmainlem}
Let $D$ be as above, and assume furthermore that $D$ has some column containing more than one non-zero element.
Then the contribution from $D$ to \eqref{Zmoments} tends to zero as $n\to\infty$.
\end{lem}
\begin{proof}
Recall that Lemma \ref{BASICBOUNDlem3} is valid for 
$\mu_1',\ldots,\mu_r'$ an arbitrary permutation of $\mu_1,\ldots,\mu_r$.
We now fix the choice of $\mu_1',\ldots,\mu_r'$ so that the number of non-zero elements in column number $\mu_1'$ is
as large as possible.
Then $a_1=\#A_1=\#\oA_1\geq\#\oA_j\geq a_j$ for all $j\in\{1,\ldots,r\}$, and $a_1\geq2$ by our %special 
assumption on $D$.

Now note that $\log(\tV_x)$, which we take to be defined for arbitrary real $x\geq1$
through the formula \eqref{TVdef}, is a strictly decreasing and strictly convex function of $x\geq1$.
This is easily verified by differentiation.
It follows that for any $j\geq2$, if $a_j\geq2$ (and thus $a_1\geq a_j\geq2$), the product
$\prod_{j=1}^r\tV_{a_j}$ \textit{increases} if we simultaneously replace 
$a_1$ by $a_1+1$ and $a_j$ by $a_j-1$.
Repeating this operation for as long as possible, and recalling $\tV_1=\tV_0=1$, we conclude that
$\prod_{j=1}^r\tV_{a_j}\leq\tV_a$
for some integer $a\geq a_1\geq2$ satisfying $a+r-1\geq m$, i.e.\ $a\geq 2m-k+1$.
Hence, applying Lemma \ref{BASICBOUNDlem3} and dividing through  by $f(n)^{k/2}$,
we conclude that the contribution from $D$ to \eqref{Zmoments} is
\begin{align*}
\ll_m n^{m(m+3)/4}f(n)^{m-k/2} \tV_a^n.
\end{align*}
If $m\leq k/2$, then this bound obviously tends to zero as $n\to\infty$,
since $\tV_a<1$ and $f(n)\to\infty$;
hence from now on we assume that $m>k/2$.
Then, using the assumption $f(n)=O(e^{cn})$ 
and the fact that $\tV_a$ is a decreasing function of $a$,
we see that our term is
$\ll_m n^{m(m+3)/4} \exp\bigl((c(m-k/2)+\log\tV_{2m-k+1})n\bigr)$,
and hence to complete the proof of the lemma it suffices to prove that
\begin{align}\label{OPTIMALFNGROWTHmainlempf1}
c<\frac{-2\log\tV_{2m-k+1}}{2m-k}.
\end{align}
However, by what we noted above, $-2\log\tV_{x+1}$ is a strictly concave function of $x\geq0$, taking the value 
$0$ at $x=0$.
Also $2m-k\leq k-2$.
Hence
\begin{align*}
\frac{-2\log\tV_{2m-k+1}}{2m-k}\geq\frac{-2\log\tV_{k-1}}{k-2}=c_k
\end{align*}
(cf.\ \eqref{CKDEF} and \eqref{TVdef}),
and so \eqref{OPTIMALFNGROWTHmainlempf1} follows from the assumption that $c<c_k$.
\end{proof}

The matrices $D$ not covered by Lemma \ref{OPTIMALFNGROWTHmainlem} are very easy to handle:
\begin{lem}\label{OPTIMALFNGROWTHtrivlem}
Let $D$ be a matrix appearing in \eqref{Zmoments} with exactly one non-zero element in each column.
Then either $D$ is accounted for in $M_{k,n}$
(cf.\ \eqref{Mkndef}) or else the contribution from $D$ to \eqref{Zmoments} tends to zero 
%(with exponential rate) 
as $n\to\infty$.
\end{lem}
\begin{proof}
Let $\Delta_i:=\max(|d_{i1}|,|d_{i2}|,\ldots,|d_{ik}|)$ for $i=1,\ldots,m$.
Then, using \cite[Remark 1]{poisson}, we obtain
\begin{align}\notag
\frac1{(2f(n))^{k/2}}\,\Big(\frac{e_1}{q}\cdots\frac{e_m}{q}\Big)^n
\int_{\R^n}\cdots\int_{\R^n}\prod_{j=1}^k\chi_n\Big(\sum_{i=1}^m \frac{d_{ij}}q\vecx_i\Big)\,d\vecx_1\ldots d\vecx_m
\hspace{50pt}
\\\label{OPTIMALFNGROWTHtrivlemPF1}
\leq %\frac1{(2f(n))^{k/2}}
f(n)^{-k/2}q^{-n}\prod_{i=1}^m\biggl(\int_{\R^n}\chi_n\Bigl(\frac{\Delta_i}{q}\vecx_i\Bigr)\,d\vecx_i\biggr)
=q^{-n}f(n)^{m-k/2}\prod_{i=1}^m\Bigl(\frac q{\Delta_i}\Bigr)^n.
\end{align}
%where we used \cite[Remark 1]{poisson}.
Now note that $k\geq2m$, since $D$ has exactly one non-zero element in each column
but at least two non-zero entries in each row.
Hence, if we keep $n$ so large that $f(n)\geq1$, we have $f(n)^{m-k/2}\leq1$.
Note also that $\Delta_i\geq q$ for each $i$, since $D$ is $\langle k,q\rangle$-admissible.
Furthermore, assuming that $D$ is \textit{not} accounted for in $M_{k,n}$, we
have either $q\geq2$ or $q=1$ at the same time as $\Delta_i>1$ for some $i$.
Hence the bound in \eqref{OPTIMALFNGROWTHtrivlemPF1} is $\leq 2^{-n}$, and the lemma is proved.
\end{proof}

%\begin{remark}
%Inspecting the proofs of Lemma \ref{OPTIMALFNGROWTHmainlem}, \ref{OPTIMALFNGROWTHtrivlem} and \ref{ALLOTHERMATRICESlem}, 
%we see that the convergence to zero in fact takes place at an exponential rate.
%\end{remark}

%We now complete the proof of Theorem \ref{OPTIMALFNGROWTHTHM}.
%
%We have to prove that $\lim_{n\to\infty}\mathbb E\big(Z_n^{\:k}\big)$ exists and equals $0$ if $k$ is odd,
%$(k-1)!!$ if $k$ is even.
%We use the formula for $\mathbb E\big(Z_n^{\:k}\big)$ in \eqref{Zmoments}.
\begin{proof}[Proof of Theorem \ref{OPTIMALFNGROWTHTHM}]
Taken together, Lemma \ref{ALLOTHERMATRICESlem} and Lemmas \ref{BASICBOUNDlem3}--\ref{OPTIMALFNGROWTHtrivlem} show that the total contribution from all $D$ in \eqref{Zmoments} which are not accounted for in $M_{k,n}$ tends to zero as $n\to\infty$.
On the other hand, the treatment of $M_{k,n}$ in the proof of Proposition \ref{MOMENTPROP}
applies verbatim in the present situation with a more general function $f$, %so long as $\lim_{n\to\infty}f(n)=\infty$
and shows that $\lim_{n\to\infty}(2f(n))^{-k/2}M_{k,n}$ exists and equals $0$ for $k$ odd and 
$(k-1)!!$ for $k$ even. Hence \eqref{OPTIMALFNGROWTHTHMres} holds.

%\vspace{5pt}

We now turn to the second statement of Theorem \ref{OPTIMALFNGROWTHTHM}.
Thus assume that $k\geq3$ and $c>c_k$; let $f:\Z^+\to\R^+$ be a function satisfying $f(n)\gg e^{cn}$ as $n\to\infty$,
and consider \eqref{Zmoments} with $\chi_n$ being the characteristic function of the closed ball of volume $f(n)$ centered at the origin.
Then the contribution from any matrix $D$ as in \eqref{WORSTMATRICES} to the sum in \eqref{Zmoments} equals
\begin{align}\label{OPTIMALFNGROWTHTHMpart2pf1}
2^{-k/2}f(n)^{\frac k2-1}V_n^{1-k}J_{k-1}^{(n)}[1,\ldots,1].
\end{align}
Now $c>c_k$ implies that $e^{(1-\frac k2)c}<\tV_{k-1}$
(cf.\ \eqref{CKDEF} and \eqref{TVdef});
hence by the second part of Lemma \ref{BASICJBOUNDSlem},
the expression in \eqref{OPTIMALFNGROWTHTHMpart2pf1} tends to $\infty$ as $n\to\infty$.
This completes the proof of Theorem \ref{OPTIMALFNGROWTHTHM}.
\end{proof} %\hfill$\square$

%\begin{remark}
%A more detailed study of $J_{k-1}^{(n)}[1,\ldots,1]$ leads to the asymptotics
%$J_{k-1}^{(n)}[1,\ldots,1]\asymp_k n^{(k^2-3k-2)/4}\, V_n^{k-1}\,{\tV_{k-1}}^{\,n}$.
%In particular this shows that for $k\geq 4$, the second statement of Theorem \ref{OPTIMALFNGROWTHTHM} 
%can be strengthened to hold also for $c=c_k$. However, we will not go into the details of this in the present paper.
%\end{remark}

\end{document}